\documentclass[12pt, reqno]{amsart}
\usepackage{amsmath, amstext, amsbsy, amssymb, amscd}
\usepackage{amsthm, latexsym, graphics, float}

\newtheorem{theorem}{Theorem}[section]
\newtheorem{lemma}[theorem]{Lemma}
\newtheorem{proposition}[theorem]{Proposition}

\theoremstyle{definition}
\newtheorem{definition}[theorem]{Definition}

\theoremstyle{remark}
\newtheorem{remark}[theorem]{Remark}

\newtheorem{conjecture}[theorem]{Conjecture}

\numberwithin{equation}{section}
\newcommand{\C}{ \mathbb C }

\newcommand{\Ext}{{\rm Ext}}

\newcommand{\pms}{\overline{\mathfrak M}}

\newcommand{\ms}{\mathfrak M}

\newcommand{\Pee}{\mathbb P}

\newcommand{\rev}{\widetilde {ev}_1}
\newcommand{\rf}{\tilde {f}_{1, 0}}
\newcommand{\Supp}{{\rm Supp}}
\newcommand{\Sym}{{\rm Sym}}

\newcommand{\uc}{\overline{\mathfrak U}}

\newcommand{\w}{\tilde}

\newcommand{\W}{\widetilde}

\newcommand{\Z}{ \mathbb Z}

\newenvironment{demo}[1]%
{\vskip-\lastskip\medskip
  \noindent
  {\em #1.}\enspace
  }%
{\qed\par\medskip
  }

\begin{document}
\title[$1$-point Gromov-Witten invariants]
      {$1$-point Gromov-Witten invariants of the moduli spaces of sheaves 
         over the projective plane}
\author[Wei-Ping Li]{Wei-Ping Li$^1$}
\address{Department of Mathematics, HKUST, Clear Water Bay, Kowloon, Hong
Kong} \email{mawpli@ust.hk}
\thanks{${}^1$Partially supported by the grants GRF601905 and GRF601808}

\author[Zhenbo Qin]{Zhenbo Qin$^2$}
\address{Department of Mathematics, University of Missouri, Columbia,
MO 65211, USA} \email{qinz@missouri.edu}
\thanks{${}^2$Partially supported by an NSF grant}

%\subjclass[2000]{Primary: 14C05; Secondary: 14F43, 17B65.}

\begin{abstract}
The Gieseker-Uhlenbeck morphism maps the Gieseker moduli space of stable
rank-$2$ sheaves on a smooth projective surface to the Uhlenbeck 
compactification, and is a generalization of the Hilbert-Chow morphism
for Hilbert schemes of points. When the surface is the complex projective plane,
we determine all the $1$-point genus-$0$ Gromov-Witten invariants
extremal with respect to the Gieseker-Uhlenbeck morphism. The main idea
is to understand the virtual fundamental class of the moduli space of 
stable maps by studying the obstruction sheaf and using
a meromorphic $2$-form on the Gieseker moduli space.
\end{abstract}

\maketitle
\date{}
%\tableofcontents

%%
%%
%%
%%
%%
%%
%%
%%
%%
%%
%%
%%
\section{\bf Introduction}
\label{sec:intro}

Recently there have been intensive interests in studying the quantum 
cohomology and Gromov-Witten theory of Hilbert schemes of points on 
algebraic surfaces.
% \cite{L-Q, ELQ, L-L, MO1}.
 Two main reasons are 
the connections with the Donaldson-Thomas theory of $3$-folds 
%\cite{MNOP, O-P, MO2}
 and with Ruan's Cohomological Crepant 
Resolution Conjecture.
% (see \cite{Ruan, B-G, C-Y} for various versions).
Roughly speaking, the Crepant Resolution Conjecture asserts that 
the quantum cohomology of an orbifold $Z$ coincides with
the quantum cohomology of a crepant resolution $Y$ of $Z$ after
analytic continuation and specialization of quantum parameters.
For an algebraic surface $X$, let $X^{[n]}$ be the Hilbert scheme of 
$n$-points on $X$ and $\Sym^n(X)$ be the $n$-th symmetric product of $X$.
It is well-known that $X^{[n]}$ is smooth of dimension $2n$ and 
the Hilbert-Chow morphism
$
\Phi: \,\, X^{[n]} \to \Sym^n(X)
$ is a crepant resolution of the global orbifold $\Sym^n(X)$. 
%Many aspects of the Crepant Resolution Conjecture for $\Phi$ 
%have been verified \cite{L-S, ELQ, L-L}. An axiomatization approach 
%having flavors of representation theory was proposed in \cite{Q-W}.

A natural generalization of the Hilbert-Chow morphism $\Phi$ is
the {\it Gieseker-Uhlenbeck morphism} $\Psi$ from the moduli space
of Gieseker semistable rank-$2$ torsion-free sheaves on $X$ to 
the   Uhlenbeck compactification space. 
This morphism was constructed in \cite{LJ1, Mor}, and was shown to 
be crepant \cite{LJ2, Q-Z} when the Gieseker moduli space is smooth. 
For the projective plane $X = \Pee^2$, the moduli space 
$\pms(n)$ of Gieseker semistable sheaves $V$ on $X$ with 
$c_1(V) = -1$ and $c_2(V) = n$ is a smooth irreducible projective 
variety of dimension $(4n-4)$ when $n \ge 1$. In \cite{Q-Z}, 
it is  proved that there is exactly one primitive integral class
$\mathfrak f \in H_2(\pms(n); \Z)$ contracted by 
the Gieseker-Uhlenbeck morphism $\Psi: \,\, \pms(n) \to \uc(n)$.

The goal of this paper is to determine all the $1$-point genus-$0$ 
Gromov-Witten invariants $\langle \alpha \rangle_{0, d\mathfrak f},
\, \alpha \in H^{8n-12}(\pms(n); \C)$ extremal with respect to $\Psi$ for $n\ge 3$.
When $n\ge 3$, 
the homology group $H_4(\pms(n); \C)$ is of rank $6$, and a basis
is given by $\{ \Xi_1, \ldots, \Xi_6 \}$ (see Sect.~\ref{sec:basisH4} 
for details). The Poincar\' e duals $\text{\rm PD}(\Xi_1), \ldots, 
\text{\rm PD}(\Xi_6)$ form a basis of $H^{8n-12}(\pms(n); \C)$.

\begin{theorem}  \label{intro-thm:1pt-inv}
Let $d \ge 1$ and $n \ge 3$. The Gromov-Witten invariants 
$\langle \alpha \rangle_{0, d\mathfrak f}$ for the classes
$\alpha = \text{\rm PD}(\Xi_1), \ldots, \text{\rm PD}(\Xi_6) 
\in H^{8n-12}(\pms(n); \C)$ are respectively equal to
\begin{eqnarray*}   
- {6/d^2}, \quad {12/d^2}, \quad 0, \quad 
- {6/d^2}, \quad 0, \quad 0.
\end{eqnarray*}
\end{theorem}

When $n=1$, the moduli space $\pms(n)$ is a point. When $n=2$, the fourth Betti
number $b_4$ of the moduli space $\pms(n)$ is equal to $3$ 
which is different from the case $n\ge 3$. The result for $n=2$ will 
appear elsewhere via a different method (see Remark \ref{n2}).

An interesting observation is that the $1$-point genus-$0$ Gromov-Witten 
invariants $\langle \alpha \rangle_{0, d\mathfrak f}$ are independent of 
the second Chern class $n$.

\begin{conjecture}  \label{intro-conj}
Let $d \ge 1$ and $n \ge 3$. Then the extremal genus-$0$ Gromov-Witten 
invariants $\langle \alpha_1, \ldots, \alpha_k \rangle_{0, d\mathfrak f}$
of the moduli space $\pms(n)$ are independent of $n$.
\end{conjecture}

There are two main ideas in our proof of Theorem~\ref{intro-thm:1pt-inv}. 
The first one is to determine the restriction of the obstruction sheaf 
of the Gromov-Witten theory for $\pms(n)$ to certain open subset of 
the moduli space $\pms_{0, 1}(\pms(n), d\mathfrak f)$ of stable maps.
This enables us to determine the $1$-point invariants 
$\langle \alpha \rangle_{0, d\mathfrak f}$ for the first four 
cohomology classes $\alpha = \text{\rm PD}(\Xi_1), \ldots, 
\text{\rm PD}(\Xi_4)$.

The second one is to study the support of the virtual fundamental 
class 
\begin{eqnarray*}
[\pms_{0, 1}(\pms(n), d\mathfrak f)]^{\text{vir}} \in
A_{4n-6} \big (\pms_{0, 1}(\pms(n), d\mathfrak f) \big )
\end{eqnarray*}
using the techniques developed in \cite{K-L, L-L}. 
By introducing a suitable meromorphic $2$-form $\Theta$ on 
the Gieseker moduli space $\pms(n)$, we show that
\begin{eqnarray}   \label{intro-ev1}
ev_1 \Big ( \Supp \big ( [\pms_{0, 1}(\pms(n), d\mathfrak f) 
]^{\text{vir}} \big ) \Big ) \,\, \subset \,\, 
\mathfrak T_{C_0}(n) \coprod \mathfrak U_{C_0}(n)
\end{eqnarray}
where $ev_1: \pms_{0, 1}(\pms(n), d\mathfrak f) \to \pms(n)$ 
is the evaluation map, and $\mathfrak T_{C_0}(n)$ (respectively, 
$\mathfrak U_{C_0}(n)$) is the subset of $\pms(n)$ consisting of 
all the non-locally free sheaves $V$ such that $V|_{C_0}$ contains 
torsion (respectively, $V|_{C_0}$ is torsion-free and unstable). 
This allows us to show that $\langle \alpha \rangle_{0, d\mathfrak f}
= 0$ for $\alpha = \text{\rm PD}(\Xi_5), \text{\rm PD}(\Xi_6)$.
%In fact, we have

%\begin{corollary}  \label{intro-cor:push-vfc}
%Let $d \ge 1$ and $n \ge 5$. Then, we have
%\begin{eqnarray*} 
%ev_{1*}\big [\overline {\frak M}_{0, 1}(\pms(n), d\mathfrak f) 
%\big ]^{\text{vir}}  \,\, = \,\, 1/d^2 \,\, \mathfrak T_{C_0}(n).
%\end{eqnarray*}
%\end{corollary}

This paper is organized as follows. In \S \ref{sec:gwinv},
the Gromov-Witten theory is reviewed. In \S \ref{sec:moduli},
we recall some properties of the Gieseker moduli space $\pms(n)$ 
 and the Gieseker-Uhlenbeck morphism $\Psi$. 
We study the boundary divisor of $\pms(n)$ consisting
of non-locally free sheaves in $\pms(n)$. In \S \ref{sec:basisH4}, 
the basis $\{ \Xi_1, \ldots, \Xi_6 \}$
for $H_4(\pms(n); \C)$ is constructed. In \S \ref{sec:obs},
we analyze the obstruction sheaf of the Gromov-Witten theory for 
$\pms(n)$. In \S \ref{sec:virtual}, (\ref{intro-ev1}) is proved.
In \S \ref{sec:1point},
we verify Theorem~\ref{intro-thm:1pt-inv}.
% and  Corollary~\ref{intro-cor:push-vfc}.

%\medskip\noindent
%{\bf Conventions.} Throughout the paper, unless otherwise 
%specified, (semi)stability means Gieseker (semi)stability. 
%For a smooth variety, we make no distinctions between its 
%divisors and the corresponding line bundles.

\bigskip\noindent
{\bf Acknowledgments.}  We would like to thank Jun Li for helpful discussions. We also would like to
thank the referee for suggestions which helped to remove  the condition $n\ge 5$ in  the original draft.
\section{\bf Stable maps and Gromov-Witten invariants}
\label{sec:gwinv}

Let $Y$ be a smooth projective variety.
A $k$-pointed {\it stable map} to $Y$ consists of
a complete nodal curve $D$ with $k$ distinct ordered smooth points
$p_1, \ldots, p_k$ and a morphism $\mu: D \to Y$ such that
the data $(\mu, D, p_1, \ldots, p_k)$ has only finitely many automorphisms.
In this case, the stable map is denoted by
$[\mu: (D; p_1, \ldots, p_k) \to Y]$.
For a fixed homology class $\beta \in H_2(Y, \mathbb Z)$,
let $\overline {\frak M}_{g, k}(Y, \beta)$ be the coarse moduli space
parameterizing all the stable maps $[\mu: (D; p_1, \ldots, p_k) \to Y]$
such that $\mu_*[D] = \beta$ and the arithmetic genus of $D$ is $g$.
Then, we have the evaluation map:
\begin{eqnarray}\label{evk}
ev_k\colon \overline {\frak M}_{g, k}(Y, \beta) \to Y^k  
\end{eqnarray}
defined by $ev_k([\mu: (D; p_1, \ldots, p_k) \to Y]) = 
(\mu(p_1), \ldots, \mu(p_k))$. It is known \cite{F-P, LT1, LT2,  Beh,  B-F} 
that the coarse moduli space 
$\overline {\frak M}_{g, k}(Y, \beta)$ is projective and
has a virtual fundamental class
$[\overline {\frak M}_{g, k}(Y, \beta)]^{\text{vir}} \in
A_{\frak d}(\overline {\frak M}_{g, k}(Y, \beta))$ where
\begin{eqnarray}\label{expected-dim}
\frak d = -(K_Y \cdot \beta) + (\dim (Y) - 3)(1-g) + k 
\end{eqnarray}
is the expected complex dimension of 
$\overline {\frak M}_{g, k}(Y, \beta)$, 
and $A_{\frak d}(\overline {\frak M}_{g, k}(Y, \beta))$
is the Chow group of $\frak d$-dimensional cycles in 
the moduli space $\overline {\frak M}_{g, k}(Y, \beta)$. 

The Gromov-Witten invariants are defined by using
the virtual fundamental class
$[\overline {\frak M}_{g, k}(Y, \beta)]^{\text{vir}}$.
Recall that an element
$\alpha \in H^*(Y, \mathbb C) {\buildrel\text{def}\over=}
\bigoplus_{j=0}^{2 \dim_{\mathbb C}(Y)} H^j(Y, \mathbb C)$ is 
{\it homogeneous} if $\alpha \in H^j(Y, \mathbb C)$ for some $j$; 
in this case, we take $|\alpha| = j$. 
Let $\alpha_1, \ldots, \alpha_k \in H^*(Y, \mathbb C)$
such that every $\alpha_i$ is homogeneous and
$
\sum_{i=1}^k |\alpha_i| = 2 {\frak d}.  
$
Then, we have the $k$-point Gromov-Witten invariant defined by:
\begin{eqnarray}\label{def-GW}
\langle \alpha_1, \ldots, \alpha_k \rangle_{g, \beta} \,\,
= \int_{[\overline {\frak M}_{g, k}(Y, \beta)]^{\text{vir}}}
ev_k^*(\alpha_1 \otimes \ldots \otimes \alpha_k).  
\end{eqnarray}

Next, we summarize certain properties concerning the
virtual fundamental class. To begin with, we recall that
{\it the excess dimension} is the difference between the dimension of
$\overline {\frak M}_{g, k}(Y, \beta)$ and
the expected dimension $\frak d$ in (\ref{expected-dim}).
%Let $T_Y$ stand for the tangent sheaf of $Y$.
For $0 \le i < k$,  use 
\begin{eqnarray}\label{k-to-i}
f_{k, i}: \overline {\frak M}_{g, k}(Y, \beta) \to
\overline {\frak M}_{g, i}(Y, \beta)  
\end{eqnarray}
to stand for the forgetful map
obtained by forgetting the last $(k-i)$ marked points
and contracting all the unstable components.
It is known that $f_{k, i}$ is flat when $\beta \ne 0$ and $0 \le i < k$.
The following can be found in \cite{LT1, Beh, Get, C-K}.

\begin{proposition}\label{virtual-prop} 
Let $\beta \in H_2(Y, \mathbb Z)$ and $\beta \ne 0$.
Let $e$ be the excess dimension of 
$\overline {\frak M}_{g, k}(Y, \beta)$, and $\frak M \subset \overline 
{\frak M}_{g, k}(Y, \beta)$ be a closed subscheme. Then,

{\rm (i)} $[\overline {\frak M}_{g, k}(Y, \beta)]^{\text{vir}} =
(f_{k, 0})^*[\overline {\frak M}_{g, 0}(Y, \beta)]^{\text{vir}}$;

{\rm (ii)} $[\overline {\frak M}_{g, k}(Y, \beta)]^{\text{vir}} =
c_e(R^1(f_{k+1, k})_*(ev_{k+1})^*T_Y)$ if
$R^1(f_{k+1, k})_*(ev_{k+1})^*T_Y$ is a rank-$e$ locally free sheaf
over the moduli space $\overline {\frak M}_{g, k}(Y, \beta)$;

{\rm (iii)} $[\overline {\frak M}_{g, k}(Y, \beta)]^{\text{vir}}|_{\frak M}
= c_e((R^1(f_{k+1, k})_*(ev_{k+1})^*T_Y)|_{\frak M})$ if there exists
an open subset $\mathfrak O$ of $\overline {\frak M}_{g, k}(Y, \beta)$
such that $\frak M \subset \mathfrak O$ (i.e, $\mathfrak O$ is
an open neighborhood of $\frak M$) and the restriction
$(R^1(f_{k+1, k})_*(ev_{k+1})^*T_Y)|_{\mathfrak O}$ is 
a rank-$e$ locally free sheaf over $\mathfrak O$.
\end{proposition}
\section{\bf The moduli space of stable rank-$2$ 
             sheaves on $\Pee^2$}
\label{sec:moduli}
\subsection{Some basic facts of the moduli space}
\label{subsect_basic} $\,$
\medskip

Throughout the rest of this paper, let $X = \Pee^2$ be 
the projective plane, and let $\ell$ be a line in $X$. 
For an integer $n$, let $\pms(n)$ be the 
moduli space parametrizing all Gieseker-semistable 
rank-$2$ sheaves $V$ over $X$ with 
$c_1(V) =-\ell, c_2(V) = n$.
Note that every such sheaf $V$ is actually 
slope-stable and hence is Gieseker-stable. It is well-known that,  
when $n \ge 1$, $\pms(n)$ is nonempty,  smooth, irreducible and 
rational with the expected dimension $(4n - 4)$; in addition,
a universal sheaf over $\pms(n) \times X$ exists.
By the Theorem~1 in \cite{Mar}, the cohomology groups
$H^i \big ( \pms(n); \Z \big )$ are torsion-free for
all $i$ and vanish for odd $i$. So are the homology groups 
$H_i \big ( \pms(n); \Z \big )$. Let $b_i(\pms(n))$ be 
the $i$-th Betti number of $\pms(n)$, and put
$$
p(\pms (n); q) = \sum\limits_{i=0}^{8n-8} b_i(\pms(n)) q^i.
$$ 
By the Theorem~0.1 of \cite{Yos},
$p(\pms (n); q)$ equals the coefficient of $t^n$ of the  series 
\begin{eqnarray}  
&\displaystyle{{1 \over (q^2-1) \cdot \sum_{n \in \Z} 
  q^{2n(2n-1)}t^{n^2}} \cdot \sum_{b \ge 0} \left ( 
  {q^{2(b+1)(2b+1)} \over 1 - q^{8(b+1)}t^{2b+1}} - {q^{2b(2b+5)} 
  \over 1 - q^{8b}t^{2b+1}} \right ) t^{(b+1)^2} \cdot}&
  \nonumber   \\ 
&\displaystyle{\cdot \prod_{d \ge 1} {1 \over (1-q^{4d-2}t^d)^2
  (1-q^{4d}t^d)^2(1-q^{4d+2}t^d)^2}.}&       \label{poin-s}
\end{eqnarray}
 
In the rest of the subsection, we review Str\o mme's work in \cite{Str} 
and give a basis of $H^4(\pms (n), \Bbb Z)$ in terms of the classes 
from \cite{Str}. A basis of $H_4(\pms (n), \Bbb Z)$ with geometric 
flavors will be constructed in Section \ref{sec:basisH4}.
 
Fix $n \ge 2$. Let $\mathcal E$ be 
a universal sheaf over $\pms(n) \times X$, and let $\pi_1$ and 
$\pi_2$ be the two natural projections on $\pms(n) \times X$.
For $0 \le k \le 2$, define
\begin{eqnarray*}
\mathcal A_k = 
R^1\pi_{1*}(\mathcal E \otimes \pi_2^*\mathcal O_X(-k \ell)).
\end{eqnarray*}
For $V \in \pms(n)$, denote $V \otimes \mathcal O_X(k\ell)$ by 
$V \otimes \mathcal O_X(k)$ or $V(k)$. Then,
\begin{eqnarray} 
&h^0(X, V(-k)) = h^2(X, V(-k)) = 0 \qquad \text{for} \, 
   0 \le k \le 2,&        \label{h02}  \\
&h^1(X, V) = h^1(X, V(-2)) = n-1, \quad h^1(X, V(-1)) = n&  
                           \label{h1}
\end{eqnarray}
by the Proposition~1.5 in \cite{Str}. It follows that 
the three sheaves $\mathcal A_0, \mathcal A_1, \mathcal A_2$ 
over $\pms(n)$ are locally free of rank $(n-1), n, (n-1)$ 
respectively. 

\begin{definition}   \label{eps-del}
Let $0 \le k \le 2$ and $r_k$ be the rank of $\mathcal A_k$.
Define
\begin{eqnarray*} 
\epsilon &=& c_1(\mathcal A_0) - c_1(\mathcal A_2), \\ 
\delta   &=& n \cdot c_1(\mathcal A_0) - 
               (n-1) \cdot c_1(\mathcal A_1),  \\
\tau_k   &=&2r_k \cdot c_2(\mathcal A_k) - 
               (r_k-1) \cdot c_1(\mathcal A_k)^2. 
\end{eqnarray*}
\end{definition}

Note that these classes $\epsilon, \delta, \tau_k$ are
independent of the choices of the universal sheaf $\mathcal E$ 
over $\pms(n) \times X$. Let $K_{\pms(n)}$ be the canonical class
of $\pms(n)$, $\ms(n)$ be the open subset of $\pms(n)$ 
parametrizing   stable bundles, and  
\begin{eqnarray}   \label{def-bdry}
\mathfrak B = \pms(n) - \ms(n).
\end{eqnarray}
By the Theorem in \cite{Str}, $\text{Pic}(\pms(n))$ is freely 
generated by $\epsilon$ and $\delta$, and 
\begin{eqnarray}   \label{cano-bdry}
K_{\pms(n)} = -3\epsilon, \quad 
\mathfrak B = n \epsilon - 2 \delta.
\end{eqnarray}
Also, $a \epsilon + b \delta$ is ample if and only if $a, b > 0$, 
and $\mathfrak B$ is irreducible and reduced.

It is known from \cite{Bea, E-S, Mar} that the cohomology ring
$H^*(\pms(n); \Z)$
is generated by the Chern classes of the bundles $\mathcal A_0, 
\mathcal A_1, \mathcal A_2$. It follows that $H^4(\pms(n); \Z)$ 
is the $\Z$-linear span of the six integral classes:
\begin{eqnarray}   \label{sixclass}
\epsilon^2, \,\, \epsilon \, \delta, \,\, \delta^2, \,\, \tau_0, 
\,\, \tau_1, \,\, \tau_2.
\end{eqnarray}
By (\ref{poin-s}), the rank of $H^4(\pms(n); \Z)$ is $3$ 
when $n = 2$, and is $6$ when $n \ge 3$. Therefore, 
if $n \ge 3$, then a linear basis of $H^4(\pms(n); \Z)$ 
is given by the six classes in (\ref{sixclass}).

\begin{remark}\label{n2} 
When $n=2$, the rank of $H_4(\pms(n); \Z)$ is different from that of the case $n\ge 3$. 
Therefore the construction of a basis of $H_4(\pms(n); \Z)$ in \S \ref{sec:basisH4} for $n \ge 3$
needs to be modified. However, there is a method to describe the moduli space $\pms(2)$
using the moduli spaces of stable sheaves on the Hirzebruch surface $\mathbb F_1$ and 
chamber structures, which enable us to compute all the Gromov-Witten invariants of $\pms(2)$
(instead of a special kind considered in this paper for $n \ge 3$). 
Since the method is different, the result for $n=2$ will appear elsewhere. 
\end{remark}

\subsection{The boundary and the Uhlenbeck compactification}
\label{subsect_uhlen} $\,$
\medskip

The quasi-projective variety $\ms(n)$ has a Uhlenbeck 
compactification
\begin{eqnarray}   \label{Uhlenbeck}
\uc(n) = \coprod_{0 \le i \le n-1} 
\ms(n-i) \times \text{Sym}^i(X)
\end{eqnarray}
according to \cite{Uhl, LJ1, Mor}. Moreover, there exists 
a birational morphism, called the {\it Gieseker-Uhlenbeck morphism}, 
\begin{eqnarray}   \label{def-Phi}
\Psi: \,\, \pms(n) \to \uc(n)
\end{eqnarray}
sending $V \in \pms(n)$ to the pair $(V^{**}, \eta)$ where
$V^{**}$ is the double-dual of $V$ and 
$$
\eta \,\, = \,\, \sum\limits_{x \in X} \, h^0 \big (X, (V^{**}/V)_x 
\big ) \,\, x.
$$
It follows that the boundary divisor $\mathfrak B$ in (\ref{def-bdry}) 
is contracted by $\Psi$ to the 
codimension-$2$ subset $\coprod\limits_{1 \le i \le n-1} 
\ms(n-i) \times \text{Sym}^i(X)$ in $\uc(n)$. 

%\begin{definition}  \label{Gie-Uhl}
%We define $\Psi$ to be the {\it Gieseker-Uhlenbeck morphism}.
%\end{definition}

%\begin{remark}  \label{rmk:Gie-Uhl}
%Let $(V_i; \sum_x \ell_x \, x) \in \ms(n-i) \times \text{Sym}^i(X)$.
%By the Theorem~1 in \cite{E-L}, the dimension of the fiber of $\Psi$
%over the point $(V_i; \sum_x \ell_x \, x)$ is equal to 
%$\sum_x (2\ell_x-1)$.
%\end{remark}

Let $\mathfrak B_* = \Psi^{-1}(\ms(n-1) \times X) \subset \pms(n)$
which parametrizes all $V \in \pms(n)$ sitting in exact sequences
$0 \to V \to V_1 \to \mathcal O_x \to 0$
for some bundle $V_1 \in \ms(n-1)$ and some point $x \in X$.  
It is an open dense subset of the boundary 
divisor $\mathfrak B$.

To construct a universal sheaf over $\mathfrak B_*\times X$,
let $\mathcal E_{n-1}^0$ be a universal sheaf over $\ms(n-1) \times X$.
By \cite{Q-Z}, $\mathfrak B_* \cong \Pee(\mathcal E_{n-1}^0)$.
For simplicity, write $\mathfrak B_* = \Pee(\mathcal E_{n-1}^0)$. Let 
\begin{eqnarray}   \label{def-pi}
\pi: \,\,\, \mathfrak B_* = \Pee(\mathcal E_{n-1}^0) 
\,\,\, \to \,\,\, \ms(n-1) \times X
\end{eqnarray}
be the natural projection. Let $\Delta_X$ be the diagonal 
of $X \times X$. Consider the obvious isomorphism
$\alpha: \ms(n-1) \times \Delta_X \to \ms(n-1) \times X$.
Then, we have the isomorphisms:
\begin{eqnarray}  \label{v10}
(\pi \times \text{Id}_X)^{-1}(\ms(n-1) \times \Delta_X)
  &\cong& \Pee(\alpha^*\mathcal E_{n-1}^0),    \nonumber   \\
\w \pi^*\mathcal E_{n-1}^0|_{(\pi \times 
  \text{Id}_X)^{-1}(\ms(n-1) \times \Delta_X)} 
  &\cong& \w \alpha^*\mathcal E_{n-1}^0
\end{eqnarray}
where $\w \pi: \mathfrak B_* \times X \to \ms(n-1) \times X$ is
the composition of $\pi \times \text{Id}_X$ and the map
$
\ms(n-1) \times X \times X \to \ms(n-1) \times X
$
which denotes the projection to the product of the first and 
third factors, and 
$
\w \alpha: \Pee(\alpha^*\mathcal E_{n-1}^0) \to \ms(n-1) \times X
$
is the composition of the natural projection
$\Pee(\alpha^*\mathcal E_{n-1}^0) \to \ms(n-1) \times \Delta_X$ 
and $\alpha$. A universal sheaf $\mathcal E'$ over 
$\mathfrak B_* \times X$ sits in the exact sequence
\begin{eqnarray}   \label{B*-univ}
0 \to \mathcal E' \to
\w \pi^*\mathcal E_{n-1}^0 \to \mathcal O_{(\pi \times 
\text{Id}_X)^{-1}(\ms(n-1) \times \Delta_X)}(1) \to 0.
\end{eqnarray}

The following lemma will be used in later sections. 

\begin{lemma}  \label{B*B*}
Let $\w \pi_i$ be the $i$-th projection on $\ms(n-1) \times X$,
and let 
$$
\mathfrak D_{n-1} = 2 c_1(\mathcal A_{n-1,1}^0) - 
c_1(\mathcal A_{n-1, 2}^0) - c_1(\mathcal A_{n-1, 0}^0)
$$
where $\mathcal A_{n-1, k}^0 = R^1\w \pi_{1*}
(\mathcal E_{n-1}^0 \otimes \w \pi_2^* \mathcal O_X(-k \ell))$.
Then as divisors in $\mathfrak B_*$, 
\begin{eqnarray}   \label{B*B*.0}
\mathfrak B_*|_{\mathfrak B_*} 
= -2 \, c_1 \big (\mathcal O_{\mathfrak B_*}(1) \big )
+ (\w \pi_1 \circ \pi)^*\mathfrak D_{n-1}
+ 2 \, (\w \pi_2 \circ \pi)^*\ell.
\end{eqnarray}
\end{lemma}
\begin{proof}
Since $\mathfrak B_*$ is open and dense in $\mathfrak B$,
we see from (\ref{cano-bdry}) and Definition~\ref{eps-del} that
\begin{eqnarray}   \label{B*B*.1}
\mathfrak B_*|_{\mathfrak B_*} 
= \mathfrak B|_{\mathfrak B_*}
= \big (2(n-1) c_1(\mathcal A_1) - n c_1(\mathcal A_2) 
- n c_1(\mathcal A_0)\big )|_{\mathfrak B_*}.
\end{eqnarray}

Next, let $0 \le k \le 2$, and let $\pi_i$ be 
the $i$-th projection on $\mathfrak B_* \times X$.
Then the restriction of $\pi_1$ to the subset 
$(\pi \times \text{Id}_X)^{-1}(\ms(n-1) \times \Delta_X)$
is an isomorphism from $(\pi \times \text{Id}_X)^{-1}
(\ms(n-1) \times \Delta_X)$ to $\mathfrak B_*$.
Thus tensoring (\ref{B*-univ}) by $\pi_2^*\mathcal O_X(-k \ell)$
and then applying the functor $\pi_{1*}$, 
we get the exact sequence
\begin{eqnarray*}   
&0 \to \mathcal O_{\mathfrak B_*}(1) \otimes 
  (\w \pi_2 \circ \pi)^*\mathcal O_X(-k \ell) \to 
  R^1\pi_{1*}(\mathcal E' \otimes \pi_2^*\mathcal O_X(-k \ell))&\\
&\to R^1\pi_{1*} (\w \pi^*\mathcal E_{n-1}^0 \otimes 
  \pi_2^*\mathcal O_X(-k \ell)) \to 0.&
\end{eqnarray*}
Note that $\pi_2 = \w \pi_2 \circ \w \pi$. Also,
$R^1\pi_{1*} \w \pi^* \cong
(\w \pi_1 \circ \pi)^*R^1\w \pi_{1*}$ via the trivial base change:
\begin{eqnarray*}
\begin{array}{ccc}
\mathfrak B_* \times X   
         &\overset {\pi_1} \longrightarrow&\mathfrak B_*\\
{\quad}\downarrow{\w \pi}
         &&{\qquad}\downarrow{\w \pi_1 \circ \pi}\\
\ms(n-1) \times X   
         &\overset {\w \pi_1}\longrightarrow&\ms(n-1).\\
\end{array}
\end{eqnarray*}
Therefore, rewriting the 3rd term in the above exact sequence, 
we obtain
\begin{eqnarray*}   
&0 \to \mathcal O_{\mathfrak B_*}(1) \otimes 
  (\w \pi_2 \circ \pi)^*\mathcal O_X(-k \ell) \to 
  R^1\pi_{1*}(\mathcal E' \otimes \pi_2^*\mathcal O_X(-k \ell))& \\
&\to (\w \pi_1 \circ \pi)^*\mathcal A_{n-1, k}^0 \to 0.&
\end{eqnarray*}
So the first Chern class $c_1 \big (R^1\pi_{1*}(\mathcal E' \otimes 
\pi_2^*\mathcal O_X(-k \ell)) \big )$ equals
\begin{eqnarray}   \label{B*B*.2}
c_1 \big (\mathcal O_{\mathfrak B_*}(1) \big )
- k (\w \pi_2 \circ \pi)^*\ell
+ (\w \pi_1 \circ \pi)^*c_1 \big ( \mathcal A_{n-1, k}^0 \big ),
\end{eqnarray}
and the second Chern class $c_2 \big (R^1\pi_{1*}(\mathcal E' 
\otimes \pi_2^*\mathcal O_X(-k \ell)) \big )$ is equal to
\begin{eqnarray}   \label{B*B*.3}
\left ( c_1 \big (\mathcal O_{\mathfrak B_*}(1) \big )
- k (\w \pi_2 \circ \pi)^*\ell \right ) \cdot 
(\w \pi_1 \circ \pi)^*c_1 \big ( \mathcal A_{n-1, k}^0 \big )
+ \, (\w \pi_1 \circ \pi)^*c_2 \big ( \mathcal A_{n-1, k}^0 \big ).
\end{eqnarray}

Finally, we conclude from (\ref{B*B*.1}) and (\ref{B*B*.2}) that
\begin{eqnarray*}   
\mathfrak B_*|_{\mathfrak B_*} 
= -2 \, c_1 \big (\mathcal O_{\mathfrak B_*}(1) \big )
+ 2 \, (\w \pi_2 \circ \pi)^*\ell 
+ (\w \pi_1 \circ \pi)^*\mathfrak D
\end{eqnarray*}
where $\mathfrak D = 2(n-1) c_1(\mathcal A_{n-1, 1}^0) - 
n c_1(\mathcal A_{n-1, 2}^0) - n c_1(\mathcal A_{n-1, 0}^0)$.
Let $\epsilon_1, \delta_1 \in \text{Pic}(\pms(n-1))$ be
the counter-parts of $\epsilon, \delta \in \text{Pic}(\pms(n))$.
By (\ref{cano-bdry}) and Definition~\ref{eps-del},
\begin{eqnarray*}
   0 
&=&\big ( \pms(n-1) - \ms(n-1) \big )|_{\ms(n-1)}  \\
&=&[(n-1)\epsilon_1 - 2 \delta_1]|_{\ms(n-1)}  \\
&=&2(n-2) c_1(\mathcal A_{n-1, 1}^0) 
     - (n-1) c_1(\mathcal A_{n-1, 2}^0) 
     - (n-1) c_1(\mathcal A_{n-1, 0}^0).
\end{eqnarray*}
So $\mathfrak B_*|_{\mathfrak B_*} 
= -2 \, c_1 \big (\mathcal O_{\mathfrak B_*}(1) \big )
+ 2 \, (\w \pi_2 \circ \pi)^*\ell
+ (\w \pi_1 \circ \pi)^*\mathfrak D_{n-1}$.
\end{proof}

\subsection{Curves in $\pms(n)$}
\label{subsect_crepant} $\,$

\medskip
We shall construct two curves in the Gieseker moduli space 
$\pms (n)$ which freely generate the homology group 
$H_2(\pms (n), \Bbb Z)$. One such curve is a fiber 
$\mathfrak f$ of the morphism $\pi$ from (\ref{def-pi}). 
The following is the Lemma~3.2 in \cite{Q-Z}.

\begin{lemma}  \label{fiber}
Let $N_{\mathfrak f \subset \pms(n)}$ be the normal 
bundle of $\mathfrak f$ in $\pms(n)$. Then,
\begin{enumerate}
\item[(i)] $\mathfrak f \cdot K_{\pms(n)} = 0$ and 
$\mathfrak f \cdot \mathfrak B = -2$;

\item[(ii)] $N_{\mathfrak f \subset \pms(n)} \cong 
\mathcal O_{\mathfrak f}^{\oplus (4n-6)} \oplus 
\mathcal O_{\mathfrak f}(-2)$;

\item[(iii)] $T_{\pms(n)}|_{\mathfrak f} \cong
\mathcal O_{\mathfrak f}^{\oplus (4n-6)} \oplus
\mathcal O_{\mathfrak f}(-2) \oplus
\mathcal O_{\mathfrak f}(2)$.              \qed
\end{enumerate}
\end{lemma}

Next, we shall construct the other curve. 
Let $n \ge 3$, and let $\xi$ consist of $n$ distinct points 
in general position in $X = \Pee^2$. If $V$
sits in a nontrivial extension
\begin{eqnarray}    \label{1Vxi}
0 \to \mathcal O_X(-1) \to V \to I_\xi \to 0,
\end{eqnarray}
then $V$ is stable and hence $V \in \pms(n)$. Moreover, since
$H^0(X, I_\xi \otimes \mathcal O_X(1)) = 0$,
the injection $\mathcal O_X(-1) \to V$ is unique up to scalars.
It follows that
\begin{eqnarray}  \label{def-E}
\mathfrak E_n 
= \Pee \big ( \text{Ext}^1(I_\xi, \mathcal O_X(-1)) \big )
\cong \Pee^{n-1}
\end{eqnarray}
can be regarded as the subset of $\pms(n)$ parametrizing all the 
sheaves $V \in \pms(n)$ sitting in nontrivial extensions 
(\ref{1Vxi}). A universal sheaf $\mathcal E'$ over $\mathfrak E_n 
\times X$ sits in 
\begin{eqnarray}  \label{E-univ}
0 \to \pi_2^*\mathcal O_X(-\ell) \to \mathcal E' \to \pi_1^*
\mathcal O_{\mathfrak E_n}(-1) \otimes \pi_2^*I_\xi \to 0.
\end{eqnarray}
Tensoring by $\pi_2^*\mathcal O_X(-k \ell)$ and applying 
$\pi_{1*}$ lead to the exact sequence:
$$0 \to 
R^1\pi_{1*}(\mathcal E' \otimes \pi_2^*\mathcal O_X(-k \ell)) 
\to \mathcal O_{\mathfrak E_n}(-1)^{\oplus h^1(X, I_\xi(-k))} 
\to \mathcal O_{\mathfrak E_n}^{\oplus 
h^2(X, \mathcal O_X(-(k+1)\ell))} \to 0$$
where $0 \le k \le 2$. An easy computation gives rise to 
the following:
\begin{eqnarray*}   
c_1 \big ( R^1\pi_{1*}\mathcal E' \big ) 
  &=& -(n-1) \cdot c_1 \big (\mathcal O_{\mathfrak E_n}(1) \big ), \\
%      \label{c1E'.2}   \\
%c_2 \big ( R^1\pi_{1*}\mathcal E' \big ) 
%  &=& (n-1)(n-2)/2 \cdot c_1 \big (\mathcal O_{\mathfrak E_n}(1) 
%      \big )^2,       \nonumber     \\
c_1 \big ( R^1\pi_{1*}(\mathcal E' \otimes 
      \pi_2^*\mathcal O_X(-k \ell)) \big ) 
  &=& -n \cdot c_1 \big (\mathcal O_{\mathfrak E_n}(1) \big ) 
%      \label{c1E'.1}   \\ 
%c_2 \big ( R^1\pi_{1*}(\mathcal E' \otimes 
%      \pi_2^*\mathcal O_X(-k \ell)) \big ) 
%  &=& n(n-1)/2 \cdot c_1 \big (\mathcal O_{\mathfrak E_n}(1) \big )^2, 
%      \nonumber   
\end{eqnarray*}
where $k = 1, 2$. It follows immediately from 
Definition~\ref{eps-del} that 
\begin{eqnarray}   
\epsilon|_{\mathfrak E_n} 
  = c_1 \big (\mathcal O_{\mathfrak E_n}(1) \big ), 
  \quad \delta|_{\mathfrak E_n} = 0.   \label{restr2}  
%&\tau_0|_{\mathfrak E_n} = \tau_1|_{\mathfrak E_n} = 0,
%  \quad \tau_2|_{\mathfrak E_n} = n \cdot 
%  c_1 \big (\mathcal O_{\mathfrak E_n}(1) \big )^2.& \label{restr2'}
\end{eqnarray}

\begin{lemma}  \label{h2-basis}
Let $n \ge 3$ and $\mathfrak l$ be a line in the projective 
space $\mathfrak E_n$. 
\begin{enumerate}
\item[(i)] The homology group $H_2(\pms(n); \Z)$ is freely generated by 
$\mathfrak f$ and $\mathfrak l$;

\item[(ii)] The class $a \mathfrak f + b \mathfrak l \in 
H_2(\pms(n); \Z)$ is effective if and only if $a, b \ge 0$;

\item[(iii)] If $C \subset \pms(n)$ is 
an irreducible curve contracted by the Gieseker-Uhlenbeck 
morphism $\Psi$, then $C = d \mathfrak f \in H_2(\pms(n); \Z)$ for 
some positive integer $d$.
\end{enumerate}
\end{lemma}
\begin{proof}
(i) Since $\pms(n)$ is rational, $H^i \big ( \pms(n), 
\mathcal O_{\pms(n)} \big ) = 0$ for all $i \ge 1$. Hence 
$
H^2 \big ( \pms(n); \Z) \cong \text{Pic}(\pms(n))$, 
and $H^2 \big ( \pms(n); \Z)$ is freely generated by 
$\epsilon$ and $\delta$. By (\ref{restr2}), $\mathfrak l 
\cdot \epsilon = 1$ and $\mathfrak l \cdot \delta = 0$. 
By Definition~\ref{eps-del} and (\ref{B*B*.2}), 
$\mathfrak f \cdot \epsilon = 0$ and 
$\mathfrak f \cdot \delta = 1$.
Since $H_2(\pms(n); \Z)$ is torsion-free,
$H_2(\pms(n); \Z)$ is freely generated by 
$\mathfrak f$ and $\mathfrak l$.

(ii) Since $\mathfrak f$ and $\mathfrak l$ are effective,
$a \mathfrak f + b \mathfrak l$ is effective if 
$a, b \ge 0$. Conversely, if $a \mathfrak f + b \mathfrak l$
is effective, then $a, b \ge 0$ since the divisor 
$c \epsilon + d \delta$ is ample if and only if $c, d > 0$.

(iii) By (ii), $C = d \mathfrak f + b \mathfrak l 
\in H_2(\pms(n); \Z)$ where $d$ and $b$ are nonnegative integers
not both zero. Let $L$ be a very ample divisor on $\uc(n)$.
Then, $\mathfrak f \cdot \Psi^*L = 0$ and $\mathfrak l \cdot 
\Psi^*L \ge 0$. Since $\Psi^*L$ is a nonzero divisor and
$H_2(\pms(n); \Z)$ is freely generated by $\mathfrak f$ and 
$\mathfrak l$, we must have $\mathfrak l \cdot \Psi^*L > 0$. 
Thus, $C \cdot \Psi^*L = 0$ forces $b = 0$.
\end{proof}

\section{\bf A basis of $H_4(\pms(n); \C)$}
\label{sec:basisH4}

In this section, we assume $n \ge 3$.
Then the integral homology group $H_4(\pms(n); \Z)$ is free 
of rank $6$. In the following, we construct a basis 
$\{ \Xi_1, \ldots, \Xi_6\}$ for $H_4(\pms(n); \C)$.
%As in previous sections, $C_0$ denotes a smooth cubic in  $X = \Pee^2$.
This construction makes use of a result due to Hirschowitz and Hulek.
%
%
%
%
%
%
%
%\subsection{Stable rank-$2$ bundles from \cite{H-H}}
%\label{subsect_hh} $\,$
%\medskip

We review the results in \cite{H-H} where 
complete rational curves were found in $\ms(n)$. 
Let $n \ge 2$ and $\Gamma = \Pee^1$.
Fix lines $\ell_1, \ldots, \ell_{n} \subset X = \Pee^2$
in general position. For $1 \le i \le n$, let 
$\phi_i: \Gamma \to \ell_i$ be an isomorphism, and define
$Y_i \subset \Gamma \times X$ to be the graph of $\phi_i$.
For generic choices of $\phi_1, \ldots, \phi_{n}$, 
it was proved in \cite{H-H} that $Y_1, \ldots, Y_{n}$ 
are disjoint.
%and $\phi_1(p), \ldots, \phi_{n}(p)$  are not colinear for any $p \in \Gamma$. 
Moreover, 
if $N_{Y/\Gamma \times X}$ denotes the normal bundle of
$Y \overset{\rm def}= \coprod_{i=1}^{n} Y_i$ in 
$\Gamma \times X$, then
$\det(N_{Y/\Gamma \times X}) 
\cong \mathcal O_{\Gamma \times X}(2, 1)|_Y$.
Therefore, the element 
$$1 \in H^0(Y, \mathcal O_Y) \cong
\Ext^1 \big (\mathcal O_{\Gamma \times X}(2, 0) 
\otimes I_Y, \mathcal O_{\Gamma \times X}(0, -1) \big )$$
defines a rank-$2$ bundle $\W{\mathcal E}_n$ over 
$\Gamma \times X$ sitting in an exact sequence
\begin{eqnarray}   \label{exact-P1}
0 \to \mathcal O_{\Gamma \times X}(0, -1) \to 
\W{\mathcal E}_n \to \mathcal O_{\Gamma \times X}(2, 0) 
\otimes I_Y \to 0
\end{eqnarray}
such that $\W{\mathcal E}_n|_{\{p\} \times X} \in \ms(n)$ for 
all $p \in \Gamma$, and $\W{\mathcal E}_n$ induces a non-constant
% an injective 
morphism 
$$
\iota: \Gamma \to \ms(n).
$$ 
Let $\w{\pi}_1$ and $\w{\pi}_2$ be the natural projections on 
$\Gamma \times X$. Let 
\begin{eqnarray}   \label{WAnk} 
\W{\mathcal A}_{n, k} = R^1\w \pi_{1*}
(\W{\mathcal E}_n \otimes \w \pi_2^* \mathcal O_X(-k \ell)).
\end{eqnarray}
By the Lemma~3.5 in \cite{H-H}, the degrees 
of the bundles $\W{\mathcal A}_{n, k}$ are 
\begin{eqnarray}  \label{ank}
   a_{n, k}
&=&\left\{
     \everymath{\displaystyle}
      \begin{array}{ll}
       2n-2,       &{\rm if} \;k=0, \\
       n,          &{\rm if} \;k=1, \\
       0,          &{\rm if} \;k=2. 
      \end{array}
   \right.
\end{eqnarray}

\begin{remark}\label{base-change}
(i) Let $n \ge 2$. By (\ref{h1}), the condition in Corollarie (6.9.9) of \cite{Gro} is satisfied.
Hence the base-change theorem of the first direct image holds for every projective morphism to $\pms(n)$
and for the sheaf $\mathcal E\otimes \pi^*_2\mathcal O_X(-k\ell)$ with $0\le k\le 2$.  
For a finite morphism $\varphi $ from $Y$ onto a subvariety of $\pms(n)$, 
the intersection numbers of $\varphi(Y)$ with $\epsilon^2, \epsilon \cdot \delta, \delta^2,
\tau_0, \tau_1, \tau_2$ can be computed on $Y$, via the projection formula,  
by pulling back the Chern classes of the first direct images of the sheaf 
$\mathcal E\otimes \pi^*_2\mathcal O_X(-k\ell)$ with $0\le k\le 2$. 
%This applies to \S \ref{subsect_Xi34} below for the subvarieties of $W$ there.  

(ii) Let $n \ge 3$. Then $(n-1) \ge 2$. In Subsect.~\ref{subsect_Xi34}, we will construct 
a surface $Y=\pi^{-1}(\Gamma\times \{y\})$  admitting a finite morphism onto a surface $\Xi_3$ of $\pms(n)$. 
This finite morphism factors through the non-constant morphism $\iota: \Gamma \to \ms(n-1)$.
Applying (i), we may assume for convenience that  $\Gamma \subset \ms(n-1)$ and consequently $Y=\Xi_3 \subset \pms(n)$
 (in fact, it can be proved that when $n \ge 5$, the morphism 
$\iota: \Gamma \to \ms(n-1)$ is injective). Similar discussion works for $\Xi_4$.
\end{remark}
\subsection{The homology classes $\Xi_1, \Xi_2
\in H_4(\pms(n); \C)$}
\label{subsect_Xi12} $\,$
\medskip

%Let $\Gamma = \Pee^1$, and let $\W{\mathcal E}_{n-1}$ be 
%the rank-$2$ bundle over $\Gamma \times X$ from (\ref{WEn}):
%\begin{eqnarray}   \label{exact-P1}
%0 \to \mathcal O_{\Gamma \times X}(0, -1) \to 
%\W{\mathcal E}_{n-1} \to \mathcal O_{\Gamma \times X}(2, 0) 
%\otimes I_Y \to 0
%\end{eqnarray}
%where $Y = \coprod_{i=1}^{n-1} Y_i$. 
%By the discussions above,
%Subsect.~\ref{subsect_hh},
% we can assume that 
%$\W{\mathcal E}_{n-1}|_{\{p \} \times X} \in \ms(n-1)$ for
%every $p \in \Gamma$ and $\W{\mathcal E}_{n-1}$ induces 
%an injective morphism 
%$  
%\iota: \Gamma \to \ms(n-1)$.
Let $n \ge 3$, and assume $\Gamma \subset \ms(n-1)$ as pointed out in Remark~\ref{base-change}~(ii). 
Let
%\begin{eqnarray}   \label{def:P}
$\Pee = \Pee(\W{\mathcal E}_{n-1}) = \pi^{-1}(\Gamma \times X)
\subset \mathfrak B_*$.
%\end{eqnarray}
Then $\Pee$ parametrizes all the sheaves $V \in \pms(n)$ sitting in
$$
0 \to V \to \W{\mathcal E}_{n-1}|_{\{p\} \times X} \to
\mathcal O_x \to 0
$$
for some $p \in \Gamma$ and $x \in X$. We still use
$\pi$ to denote the natural projection
$$
\pi: \,\, \Pee \to \Gamma \times X.
$$
By (\ref{B*-univ}), a universal sheaf $\W{\mathcal E}$ 
over $\Pee \times X$ sits in the exact sequence
\begin{eqnarray}   \label{V'}
0 \to \W{\mathcal E} \to
\w \pi^*\W{\mathcal E}_{n-1} \to \mathcal O_{(\pi \times 
\text{Id}_X)^{-1}(\Pee \times \Delta_X)}(1) \to 0
\end{eqnarray}
where $\w \pi$ is the composition of
$\pi \times \text{Id}_X: \Pee \times X \to 
\Gamma \times X \times X$ and the projection 
$\Gamma \times X \times X \to \Gamma \times X$
to the first and third factors.

Let $\pi_i$ (respectively, $\w{\pi}_i$) be the natural 
projections on $\Pee \times X$ (respectively, 
$\Gamma \times X$). By (\ref{B*B*.2}) and (\ref{B*B*.3}),
$c_1 \big (R^1\pi_{1*}(\W{\mathcal E} \otimes 
\pi_2^*\mathcal O_X(-k \ell)) \big )$ equals
\begin{eqnarray}   \label{Pc1}
c_1 \big (\mathcal O_{\Pee}(1) \big )
- k (\w \pi_2 \circ \pi)^*\ell
+ (\w \pi_1 \circ \pi)^*c_1 \big ( \W{\mathcal A}_{n-1, k} \big )
\end{eqnarray}
where $\W{\mathcal A}_{n-1, k}$ is defined in (\ref{WAnk}),
and $c_2 \big (R^1\pi_{1*}(\W{\mathcal E} 
\otimes \pi_2^*\mathcal O_X(-k \ell)) \big )$ equals
\begin{eqnarray}   \label{Pc2}
\left ( c_1 \big (\mathcal O_{\Pee}(1) \big )
- k (\w \pi_2 \circ \pi)^*\ell \right ) \cdot 
(\w \pi_1 \circ \pi)^*c_1 \big ( \W{\mathcal A}_{n-1, k} \big ).
\end{eqnarray}
In addition, we conclude from (\ref{exact-P1}) that 
\begin{eqnarray}    \label{O1square}
   c_1 \big (\mathcal O_{\Pee}(1) \big )^2 
&=&\pi^*c_1 \big ( \W{\mathcal E}_{n-1} \big ) \cdot
   c_1 \big (\mathcal O_{\Pee}(1) \big )
   - \pi^*c_2 \big ( \W{\mathcal E}_{n-1} \big )  \nonumber \\
&=&\pi^*(2, -1) \cdot c_1 \big (\mathcal O_{\Pee}(1) \big )
   - \pi^* \left [ (2, 0) \cdot 
  (0, -1) + \sum_{i=1}^{n-1} Y_i \right ]
\end{eqnarray}

Fix a point $p \in \Gamma$ and let $V_1 = 
\W{\mathcal E}_{n-1}|_{\{p\} \times X}$. Consider $\Pee(V_1)
\subset \Pee \subset \pms(n)$.
By Definition~\ref{eps-del},  (\ref{Pc1}), (\ref{Pc2}) and
(\ref{O1square}), we obtain
\begin{eqnarray}  
\epsilon|_{\Pee(V_1)} &=& 2 \pi^*\ell,   \nonumber  \\
\delta|_{\Pee(V_1)} &=& c_1 \big (\mathcal O_{\Pee(V_1)}(1) 
  \big ) + (n-1) \, \pi^*\ell,    \nonumber    \\
\tau_k|_{\Pee(V_1)} &=& -(r_k-1) \big (c_1 \big (
  \mathcal O_{\Pee(V_1)}(1) \big ) - k \, \pi^*\ell \big )^2 
  \nonumber     \\
c_1 \big (\mathcal O_{\Pee(V_1)}(1) \big )^2 &=&
  -\pi^*\ell \cdot c_1 \big (\mathcal O_{\Pee(V_1)}(1) \big )
  - (n-1) \pi^*x    \label{Bsquare}
\end{eqnarray}
where by abusing notations, $\pi$ denotes the natural 
projection $\Pee(V_1) \to X$.

Define $\Xi_1 \in H_4(\pms(n); \C)$ to be the surface
$\pi^{-1}(\{p\} \times \ell) \subset \Pee(V_1) \subset \Pee 
\subset \pms(n)$ where $\ell \subset X = \Pee^2$ is 
a fixed line. Then we have
\begin{eqnarray}   \label{inter-Xi_1}
&\epsilon^2 \cdot \Xi_1 = 0, \quad 
  \epsilon  \cdot \delta \cdot \Xi_1 = 2, \quad 
  \delta^2 \cdot \Xi_1 = 2n-3,&   \nonumber  \\ 
&\tau_0 \cdot \Xi_1 = n-2, \quad
  \tau_1 \cdot \Xi_1 = 3n-3, \quad 
  \tau_2 \cdot \Xi_1 = 5n - 10.&
\end{eqnarray}

Next, define $\Xi_2 \in H_4(\pms(n); \C)$
to be $c_1 \big ( \mathcal O_{\Pee(V_1)}(1) \big )$ regarded 
as a $2$-dimensional cycle in $\pms(n)$ via the inclusion 
$\Pee(V_1) \subset \pms(n)$. Then
\begin{eqnarray}   \label{inter-Xi_2}
&\epsilon^2 \cdot \Xi_2 = 4, \quad 
  \epsilon  \cdot \delta \cdot \Xi_2 = 2n-4, \quad 
  \delta^2 \cdot \Xi_2 = n^2-5n+5,&   \nonumber  \\ 
&\tau_0 \cdot \Xi_2 = (n-2)^2, \,\,
  \tau_1 \cdot \Xi_2 = (n-1)(n-5), \,\,
  \tau_2 \cdot \Xi_2 = (n-2)(n-10). \quad&
\end{eqnarray}
\subsection{The homology classes $\Xi_3, \Xi_4
\in H_4(\pms(n); \C)$}
\label{subsect_Xi34} $\,$
\medskip

Fix a line $\ell \subset X = \Pee^2$. Let $W = \pi^{-1}
(\Gamma \times \ell) \subset \Pee \subset \pms(n)$ where 
$\Gamma$, $\Pee$ and $\pi$ are from the previous subsection.
By Definition~\ref{eps-del}, (\ref{ank}),  (\ref{Pc1}) and (\ref{Pc2}), 
\begin{eqnarray*}  
\epsilon|_{W} &=& \pi^*(2n-4, 2), \\
\delta|_{W} &=& c_1 \big (\mathcal O_{W}(1) 
  \big ) + \pi^*(n^2-2n-1, n-1), \\
\tau_k|_{W} 
%&=&-(r_k-1) c_1 \big (\mathcal O_{W}(1) \big )^2 +
%  \pi^*(2a_{n-1,k}, 2k(r_k-1)) \cdot c_1 \big (\mathcal O_{W}(1) \big ) 
%  - 2ka_{n-1,k} \, \pi^*x \\
   &=& \pi^*\big (2a_{n-1,k}-2(r_k-1), (2k+1)(r_k-1) \big ) \cdot 
        c_1 \big (\mathcal O_{W}(1) \big ) \\
& &\qquad\qquad + 
  \big ( (r_k-1)(n-3)- 2ka_{n-1,k} \big ) \, \pi^*x              
\end{eqnarray*}
where $x \in \Gamma \times \ell$ is a fixed point and 
by abusing notations, $\pi: W \to \Gamma \times \ell \cong 
\Pee^1 \times \Pee^1$ stands for the natural projection. 
In addition, by (\ref{O1square}), we conclude that
\begin{eqnarray}  \label{Bsquare2}
c_1 \big (\mathcal O_{W}(1) \big )^2 = \pi^*(2, -1) \, 
c_1 \big (\mathcal O_{W}(1) \big ) - (n-3) \, \pi^*x.
\end{eqnarray}

Define $\Xi_3 \in H_4(\pms(n); \C)$ to be the surface
$\pi^{-1}(\Gamma \times \{y\}) \subset W \subset \pms(n)$ 
where $y \in \ell$ is a fixed point. A straightforward 
computation shows that
\begin{eqnarray}   \label{inter-Xi_3}
&\epsilon^2 \cdot \Xi_3 = 0, \quad 
  \epsilon  \cdot \delta \cdot \Xi_3 = 2n-4, \quad 
  \delta^2 \cdot \Xi_3 = 2n^2-4n,&   \nonumber  \\ 
&\tau_0 \cdot \Xi_3 = 2n-4, \quad
  \tau_1 \cdot \Xi_3 = 0, \quad 
  \tau_2 \cdot \Xi_3 = -2n+4.&
\end{eqnarray}

Define $\Xi_4 \in H_4(\pms(n); \C)$
to be $c_1 \big ( \mathcal O_{W}(1) \big )$ regarded 
as a $2$-dimensional cycle in $\pms(n)$ via the inclusion 
$W \subset \pms(n)$. Then we have
\begin{eqnarray}   \label{inter-Xi_4}
&\epsilon^2 \cdot \Xi_4 = 8(n-2), \,
  \epsilon  \cdot \delta \cdot \Xi_4 = 4n^2-12n+10, \,
  \delta^2 \cdot \Xi_4 = 2n^3-8n^2+9n-1,&   \nonumber  \\ 
&\tau_0 \cdot \Xi_4 = n^2-5n+6, \,\,
  \tau_1 \cdot \Xi_4 = n^2-1, \,\,
  \tau_2 \cdot \Xi_4 = (n-2)(n+9). \quad&
\end{eqnarray}
\subsection{The homology class $\Xi_5 \in H_4(\pms(n); \C)$}
\label{subsect_Xi_5} $\,$
\medskip

Since $n \ge 3$, the moduli space $\ms(n-2)$ is nonempty. 
Fix a vector bundle $V_2 \in \ms(n-2)$ and two distinct points 
$x_1, x_2 \in X$. Let $\Xi_5 \subset \pms(n)$ parametrize all 
the sheaves $V \in \pms(n)$ sitting in exact sequences:
$$
0 \to V \to V_2 \to \mathcal O_{x_1} \oplus \mathcal O_{x_2} 
\to 0.
$$
Then we have the isomorphisms $\Xi_5 \cong \Pee(V|_{x_1}) \times 
\Pee(V|_{x_2}) \cong \Pee^1 \times \Pee^1$. Moreover, a universal 
sheaf $\mathcal E'$ over $\Xi_5 \times X$ sits in the exact sequence:
\begin{eqnarray*}
0 \to \mathcal E' \to \pi_2^*V_2 \to \w \pi_1^*
\mathcal O_{\Pee^1}(1)|_{\Pee^1 \times \Pee^1 \times \{x_1\}} 
\oplus \w \pi_2^*\mathcal O_{\Pee^1}(1)|_{\Pee^1 \times 
\Pee^1 \times \{x_2\}} \to 0
\end{eqnarray*}
where $\pi_i$ and $\w \pi_i$ denote the $i$-th projection on 
$\Xi_5 \times X$ and $\Pee^1 \times \Pee^1 \times X$
respectively. Tensoring the above exact sequence by 
$\pi_2^*\mathcal O_X(-k \ell))$ and applying $\pi_{1*}$ yield
$$
0 \to \mathcal O_{\Xi_5}(1, 0) \oplus \mathcal O_{\Xi_5}(0, 1)
\to R^1\pi_{1*}(\mathcal E' \otimes \pi_2^*\mathcal O_X(-k \ell))
\to \mathcal O_{\Xi_5}^{\oplus h^1(X, V_2(-k))} \to 0
$$
where $0 \le k \le 2$. Therefore $c_1 \big ( R^1\pi_{1*}(\mathcal E' 
\otimes \pi_2^*\mathcal O_X(-k \ell)) \big ) 
= c_1 \big (\mathcal O_{\Xi_5}(1, 1) \big )$ and
$
c_2 \big ( R^1\pi_{1*}(\mathcal E' 
\otimes \pi_2^*\mathcal O_X(-k \ell)) \big ) = 1
$
for $0 \le k \le 2$. It follows from Definition~\ref{eps-del} that
\begin{eqnarray*}
\epsilon|_{\Xi_5} = 0, \quad \delta|_{\Xi_5} 
= c_1 \big (\mathcal O_{\Xi_5}(1, 1) \big ).
\end{eqnarray*}

Regarding $\Xi_5 \in H_4(\pms(n); \C)$,
we obtain the intersection numbers on $\pms(n)$:
\begin{eqnarray}   \label{inter-Xi_5}
\epsilon^2 \cdot \Xi_5 = 
 \epsilon  \cdot \delta \cdot \Xi_5 =0, \quad 
 \delta^2 \cdot \Xi_5 = 2,\tau_0 \cdot \Xi_5 = 
\tau_1 \cdot \Xi_5 = 
%:
\tau_2 \cdot \Xi_5 = 2.
\end{eqnarray}
\subsection{The surface $\Xi_6$ in $\pms(n)$}
\label{subsect_Xi_6} $\,$
\medskip

Fix a vector bundle $V_2 \in \ms(n-2)$ and a point $x \subset X$. 
Fix a trivialization of $V_2$ in an open neighborhood $O_x$ of $x$.
Let $X^{[2]}$ be the Hilbert scheme parametrizing the length-$2$
closed subschemes of $X$, and let
$$
M_2(x) = \{ \xi \in X^{[2]}|\,\, \Supp(\xi) = \{x\} \} \cong \Pee^1.
$$
For $\xi \in M_2(x)$, let $\iota_\xi: \mathcal O_X \to \mathcal O_\xi$
be the natural quotient morphism.

Let $\Xi_6 = \Pee^1 \times M_2(x)$ be the subset of $\pms(n)$
parametrizing all the sheaves $V \in \pms(n)$ sitting in extensions
of the form
$$
0 \longrightarrow V \longrightarrow V_2 
\overset{(a, b)} \longrightarrow \mathcal O_\xi \longrightarrow 0
$$
where $\xi \in M_2(x)$, and for $(a, b) \in \Pee^1$, the map
$V_2 \overset{(a, b)} \longrightarrow \mathcal O_\xi$ denotes the
composition:
\begin{eqnarray*}
V_2 \quad \longrightarrow \quad V_2|_{O_x} \cong \mathcal O_{O_x}^{\oplus 2} 
\quad \overset{( a\iota_\xi, b \iota_\xi)} \longrightarrow \quad \mathcal O_\xi.
\end{eqnarray*}

Next, we construction a universal sheaf over $\Xi_6 \times X = \Pee^1 \times M_2(x) \times X$.
Let $\pi_{i_1, \ldots, i_m}$ denote the projection of 
$\Pee^1 \times M_2(x) \times X$ to the product
of the $i_1$-th, \ldots, $i_m$-th factors.
Over $\Xi_6 = \Pee^1 \times M_2(x)$, there is a tautological surjection
$\mathcal O_{\Xi_6}^{\oplus 2} \to \mathcal O_{\Xi_6}(1, 0) \to 0$.
This pulls back to a surjection over $\Pee^1 \times M_2(x) \times X$:
\begin{eqnarray}    \label{surj1}
\mathcal O_{\Xi_6 \times X}^{\oplus 2} \to 
\pi_{1}^*\mathcal O_{\Pee^1}(1) \to 0.
\end{eqnarray}

Let $\mathcal Z_2(x)$ be the universal codimension-$2$ subscheme
in $M_2(x) \times X$, and let
$
\mathcal O_{M_2(x) \times X} \to \mathcal O_{\mathcal Z_2(x)} \to 0
$
be the natural surjection. Pulling back to $\Pee^1 \times M_2(x) \times X$
yields a surjection:
\begin{eqnarray}    \label{surj2}
\mathcal O_{\Xi_6 \times X} \to 
\pi_{2,3}^*\mathcal O_{\mathcal Z_2(x)} \to 0.
\end{eqnarray}
Tensoring (\ref{surj1}) and (\ref{surj2}), we obtain a surjection
\begin{eqnarray}  \label{surj3}
\mathcal O_{\Xi_6 \times X}^{\oplus 2} \to \pi_{1}^*\mathcal O_{\Pee^1}(1) 
\otimes \pi_{2,3}^*\mathcal O_{\mathcal Z_2(x)} \to 0.
\end{eqnarray}
Since $\mathcal Z_2(x)$ is supported on $M_2(x) \times \{x\}$, 
in view of the trivialization of $V_2$ near $x$, % on the open neighborhood $O_x$ of $x$, 
(\ref{surj3}) induces a surjection
$
\pi_3^*V_2 \to \pi_{1}^*\mathcal O_{\Pee^1}(1) 
\otimes \pi_{2,3}^*\mathcal O_{\mathcal Z_2(x)} \to 0
$.
The kernel $\mathcal E'$ of the map $\pi_3^*V_2 \to 
\pi_{1}^*\mathcal O_{\Pee^1}(1) 
\otimes \pi_{2,3}^*\mathcal O_{\mathcal Z_2(x)}$ is a universal sheaf
over $\Xi_6 \times X$:
\begin{eqnarray*}
0 \to \mathcal E' \to \pi_3^*V_2 \to \pi_{1}^*\mathcal O_{\Pee^1}(1) 
\otimes \pi_{2,3}^*\mathcal O_{\mathcal Z_2(x)} \to 0
\end{eqnarray*}
It follows that for $0 \le k \le 2$, we have the exact sequence over $\Xi_6$:
\begin{eqnarray*}   \label{hil-line}
0 \to \w \pi_1^*\mathcal O_{\Pee^1}(1) \otimes 
\w \pi_2^*\big ( \mathcal O_X^{[2]}|_{M_2(x)} \big ) \to 
R^1(\pi_{1,2})_*(\mathcal E' \otimes \pi_3^*\mathcal O_X(-k \ell))
\to \mathcal O_{\Xi_6}^{\oplus h^1(X, V_2(-k))} \to 0
\end{eqnarray*}
where $\w \pi_1$ and $\w \pi_2$ are the two natural projections on 
$\Xi_6 = \Pee^1 \times M_2(x)$, and $\mathcal O_X^{[2]}$ is the 
tautological rank-$2$ bundle over the Hilbert scheme $X^{[2]}$
whose fiber at a point $\xi \in X^{[2]}$ is the space 
$H^0(X, \mathcal O_\xi)$. It is well-known that 
$-2c_1 \big ( \mathcal O_X^{[2]} \big ) =  M_2(X)
$
where $M_2(X) \subset X^{[2]}$ consists of all the elements 
$\xi \in X^{[2]}$ such that $|\Supp(\xi)| = 1$.
Since $M_2(X) \cdot M_2(x) = -2$, we conclude from the above exact 
sequence that
\begin{eqnarray*}
   c_1 \big ( R^1(\pi_{1,2})_*(\mathcal E' 
     \otimes \pi_3^*\mathcal O_X(-k \ell)) \big ) 
&=&(2, 1) \in \text{\rm Pic}(\Xi_6) \cong 
            \text{\rm Pic}(\Pee^1 \times \Pee^1)  \\
   c_2 \big ( R^1(\pi_{1,2})_*(\mathcal E' 
     \otimes \pi_3^*\mathcal O_X(-k \ell)) \big ) 
&=&1.
\end{eqnarray*}
By Definition~\ref{eps-del}, we get the six intersection numbers 
on $\pms(n)$:
\begin{eqnarray}   \label{inter-Xi_6}
&\epsilon^2 \cdot \Xi_6 = 0, \quad
 \epsilon  \cdot \delta \cdot \Xi_6 = 0, \quad 
 \delta^2 \cdot \Xi_6 = 4,&  \nonumber  \\
&\tau_0 \cdot \Xi_6 = -2n+6, \quad
\tau_1 \cdot \Xi_6 = -2n+4, \quad
\tau_2 \cdot \Xi_6 = -2n+6.&
\end{eqnarray}

%
%
%
%
%
%
%
%\subsection{A basis of $H_4(\pms(n); \C)$}
%\label{subsect_basisH4} $\,$
Now we can summarize the above in the following proposition.
\begin{proposition}  \label{prop:basisH4}
Let $n \ge 3$. Then $\{ \Xi_1, \ldots, \Xi_6 \}$ is a linear 
basis of $H_4(\pms(n); \C)$.
\end{proposition}
\begin{proof}
Note from (\ref{inter-Xi_1}), (\ref{inter-Xi_2}),
(\ref{inter-Xi_3}), (\ref{inter-Xi_4}), 
(\ref{inter-Xi_5}) and (\ref{inter-Xi_6})
that the intersection matrix between the classes
$\Xi_1, \ldots, \Xi_6$ and the classes 
$\epsilon^2, \epsilon\cdot
\delta, \delta^2, \tau_0, 
\tau_1, \tau_2$ has a nonzero determinant. 
Since $H_4(\pms(n); \C)$ has dimension $6$,
our result follows.
\end{proof}

\section{\bf The restriction of the obstruction sheaf 
on certain open subset}
\label{sec:obs}
From the previous section, we see that, if we let 
$\pms_*(n) = \mathfrak B_* \cup \ms(n)$, 
then the classes $\Xi_1, \Xi_2, \Xi_3, \Xi_4$ lie in $\pms_*(n)$ while 
$\Xi_5$ and $\Xi_6$ lie in the complement $\pms (n)-\pms_*(n)$.
Since a stable map $[\mu\colon (D, p)\to \pms (n)]$ in $ev_1^{-1}(\Xi_i)
\subset \pms_{0, 1}(\pms (n), d\mathfrak f)$ has $\mu(p)\in \Xi_i$ for 
$1\le i\le 4$, we have $\mu(D)\subset \mathfrak B_*$. In this section, 
we use the geometric construction of $\mathfrak B_*$ in 
Subsect.~\ref{subsect_uhlen} to effectively compute the virtual cycle 
restricted to $ev_1^{-1}(\mathfrak B_*)$. The result will be used to 
compute the Gromov-Witten invariants $\langle \alpha \rangle_{0, d\mathfrak f}$ 
when $\alpha$ is dual to the classes $\Xi_1, \Xi_2, \Xi_3, \Xi_4$. 

%Recall that $\mathfrak f$ stands for a fiber of the natural projection 
%$\pi$ in (\ref{def-pi}), and $\mathfrak f$ is contracted by 
%the Gieseker-Uhlenbeck  morphism $\Psi$.
%Let $\uc_*(n) = \Psi(\pms_*(n))$. Then $\pms_*(n)$ and $\uc_*(n)$ are 
%open subsets of $\pms(n)$ and $\uc(n)$ respectively.

Fix $d \ge 1$. Consider the open subset $\mathfrak O_0$ of 
$\pms_{0, 0}(\pms(n), d \mathfrak f)$ consisting of stable maps 
$[\mu \colon D \to \pms(n)]$ such that $\mu(D)\subset \pms_*(n)$. 
Similarly, take the open subset $\mathfrak O_1$ of  
$\pms_{0, 1}(\pms(n), d\mathfrak f)$ consisting of 
stable maps $[\mu\colon (D; p)\to \pms(n)]$ such that 
$\mu(D)\subset \pms_*(n)$. 
Clearly $\mathfrak O_1=f^{-1}_{1, 0}(\mathfrak O_0)$. 
Let $[\mu \colon (D; p) \to \pms(n)] \in \mathfrak O_1$.
Since $\mu(D) = d\mathfrak f$ in $H_2(\pms(n); \Z)$, 
$\mu(D) \subset \mathfrak B_*$ and $\mu(D)$ is a fiber of 
the projection $\pi$. Moreover, the composition $\Psi \circ ev_1$ 
sends the stable map $[\mu\colon (D; p)\to \pms(n)]$ 
to a point in $\ms(n-1) \times X \subset \uc(n)$,
which is independent of the marked point $p$ on $D$. 
Hence $ev_1$ induces a morphism $\phi$ from $\mathfrak O_0$ to 
$\ms(n-1) \times X$. Putting $\rev = ev_1|_{\mathfrak O_1}$ 
and $\rf = f_{1, 0}|_{\mathfrak O_1}$,
we have the following commutative diagram:
\begin{eqnarray}\label{com-diagram2}
\begin{matrix}
{\mathfrak O_1}&{\buildrel{\rev}\over \rightarrow }  &\mathfrak B_*
  &&\\
\quad \downarrow^{\rf}&&\downarrow^{\pi}\\
{\mathfrak O_0}&{\buildrel{\phi}\over \rightarrow }&\ms(n-1) \times X
  &{\buildrel {\w \pi_2}\over  \rightarrow}&X.       
\end{matrix}  
\end{eqnarray}

Note that the fiber $\phi^{-1}(V_1, x)$ over a point 
$(V_1, x) \in \ms(n-1) \times X$ is simply 
$$
\pms_{0, 0}\big (\pi^{-1}(V_1, x), d[\pi^{-1}(V_1, x)])
$$
which is isomorphic to $\pms_{0, 0}(\mathbb P^1, d[\mathbb P^1])$
via the isomorphism $\pi^{-1}(V_1, x) \cong \mathbb P^1$.
Hence the complex dimension of the open subset $\mathfrak O_0
\subset \pms_{0, 0}(\pms(n), d \mathfrak f)$ is equal to 
\begin{eqnarray*}
& &\dim \overline{\frak M}_{0, 0}(\mathbb P^1, d[\mathbb P^1])
   + \dim \ms(n-1) + \dim X  \\
&=&(2d-2) + [4(n-1)-4]+ 2 \\
&=&2d + 4n - 8.
\end{eqnarray*}
Since $K_{\pms(n)}\cdot d\mathfrak f=0$, the expected dimension 
of $\mathfrak O_0 \subset \pms_{0, 0}(\pms(n), d\mathfrak f)$ 
is $(4n-7)$ by (\ref{expected-dim}).
Hence the excess dimension of $\mathfrak O_0$ is $e=(2d-1)$.

Let $\mathcal V$ be the restriction of 
$R^1(f_{1, 0})_*(ev_{1})^*T_{\pms(n)}$ to $\mathfrak O_0$. 

\begin{lemma}\label{obstruction-bundle1} % \label{localyfree} 
{\rm (i)} The sheaf $\mathcal V$ is 
locally free of rank $(2d-1)$;

\medskip
{\rm (ii)} $\mathcal V \cong 
R^1(\rf)_*(\rev)^*\mathcal O_{\mathfrak B_*}(\mathfrak B_*).$
\end{lemma}
\begin{demo}{Proof}
(i) Take a stable map $u=[\mu\colon D \to \pms(n)]$ in $\mathfrak O_0$, 
and consider
$$H^1(f^{-1}_{1, 0}(u), (ev_1^*T_{\pms(n)})|_{f^{-1}_{1, 0}(u)}) 
\cong H^1(D, \mu^*T_{\pms(n)}).$$
Since $\mu(D) \cong \Pee^1$ is a fiber of the projection $\pi$, 
we see from Lemma~\ref{fiber}~(iii) that
$$
T_{\pms(n)}|_{\mu(D)} = \mathcal O_{\mu(D)}^{\oplus (4n-6)} \oplus 
\mathcal O_{\mu(D)}(-2) \oplus \mathcal O_{\mu(D)}(2).
$$ 
Thus $H^1(D, \mu^*T_{\pms(n)})\cong H^1(D, \mu^*\mathcal O_{\mu(D)}(-2))$
%\cong H^0(D, \mu^*\mathcal O_{\mathbb P^1}(2)\otimes \omega_D),
whose dimension equals the excess dimension $e=(2d-1)$. 
Therefore, the restriction $\mathcal V$ of 
$R^1(f_{1, 0})_*(ev_{1})^*T_{\pms(n)}$ to $\mathfrak O_0$
is a locally free sheaf of rank $(2d-1)$.

(ii) Since $ev_1(\mathfrak O_1)\subset \mathfrak B_*$, we have 
$((ev_1)^*T_{\pms(n)})|_{\mathfrak O_1}=
(\rev)^*(T_{\pms_*(n)}|_{\mathfrak B_*})$ and 
\begin{eqnarray*}
   \mathcal V
&=&(R^1(f_{1, 0})_*(ev_1)^*T_{\pms(n)})|_{\mathfrak O_0} 
=R^1(\tilde f_{1, 0})_*
   \big(((ev_1)^*T_{\pms(n)})|_{\mathfrak O_1}\big)  \\
&=&R^1(\tilde f_{1, 0})_*(\rev)^*(T_{\pms(n)}|_{\mathfrak B_*}).
\end{eqnarray*}
Since $\mathfrak B_*$ is smooth of codimension-$1$ in $\pms(n)$, 
we obtain the exact sequence
\begin{eqnarray}  \label{ex-tangent}
0 \to T_{\mathfrak B_*}\to T_{\pms(n)}|_{\mathfrak B_*} \to 
\mathcal O_{\mathfrak B_*}(\mathfrak B_*)\to 0.
\end{eqnarray}
Applying $(\rev)^*$ and $(\rf)_*$ to 
the exact sequence (\ref{ex-tangent}), we get 
\begin{eqnarray}  \label{r1rf}
R^1 (\rf)_*(\rev)^*T_{\mathfrak B_*} \to \mathcal V \to 
R^1 (\rf)_*(\rev)^*\mathcal O_{\mathfrak B_*}(\mathfrak B_*) \to 0.
\end{eqnarray}
where we have used $R^2(\rf)_*(\rev)^*T_{\mathfrak B_*}=0$
since $\rf$ is of relative dimension $1$.

If $[\mu\colon D \to \pms(n)]$ is a stable map in $\mathfrak O_0$, 
then $\mu(D)$ is a fiber of the projection $\pi$ in (\ref{def-pi}). 
Hence the normal bundle of $\mu(D)$ in $\mathfrak B_*$ is trivial. 
Therefore we have
$
T_{\mathfrak B_*}|_{\mu(D)} \cong 
\mathcal O_{\mu(D)}^{\oplus(4n-6)} \oplus \mathcal O_{\mu(D)}(2)$,
and 
 $H^1(D, \mu^*T_{\mathfrak B_*}) \cong 
H^1(D, \mu^*(\mathcal O_{\mu(D)}^{\oplus(4n-6)} \oplus 
\mathcal O_{\mu(D)}(2)))=0$. It follows that
$R^1 (\rf)_*(\rev)^*T_{\mathfrak B_*}=0$. %  \tag*{$\qed$}
\end{demo}

%Suppose that $\frak M_1$ is a closed subset of 
%$\overline{\frak M}_{0, 1}(\pms(n), d\mathfrak f)$ contained in 
%$\mathfrak O_1$. Let 
%\begin{eqnarray*}
%\frak M_0 = f_{1, 0}(\frak M_1). 
%\end{eqnarray*} 
%Then $\frak M_0 \subset \mathfrak O_0 \subset 
%\overline{\frak M}_{0, 0}(\pms(n), d\mathfrak f)$. 
%By Lemma~\ref{localyfree} and 
%Proposition~\ref{virtual-prop}~(i) and (iii),  
%\begin{eqnarray}        \label{virtual-formula}
%[\pms_{0, 1}(\pms(n), d\mathfrak f)]^{\text{vir}}|_{\frak M_1}
%= (\rf)^*c_{2d-1}\big((R^1(f_{1, 0})_*(ev_{1})^*
%T_{\pms(n)})|_{\frak M_0}\big).
%\end{eqnarray} 

\begin{proposition}    \label{obstruction-bundle2}
%Let $\mathcal V$ denote the restriction of 
%$R^1(f_{1, 0})_*(ev_{1})^*T_{\pms(n)}$ to $\mathfrak O_0$. 
%Let $(\mathcal E_{n-1}^0)^*$ be the dual of the universal rank-$2$
%bundle $\mathcal E_{n-1}^0$ over $\ms(n-1) \times X$. 
Put $\mathcal L = \det (\mathcal E_{n-1}^0)$. 
%\begin{eqnarray*}
%\mathcal L &=& \det (\mathcal E_{n-1}^0), \\
%\mathcal A_{n-1, k}^0 &=& R^1\w \pi_{1*}
 % (\mathcal E_{n-1}^0 \otimes \w \pi_2^* \mathcal O_X(-k \ell)), \\
%\mathfrak D_{n-1} &=& 2 c_1(\mathcal A_{n-1, 1}^0) - 
 % c_1(\mathcal A_{n-1, 2}^0) - c_1(\mathcal A_{n-1, 0}^0).
%\end{eqnarray*}
%where $0 \le k \le 2$. 
Then $\mathcal V$ sits in the exact sequence 
\begin{eqnarray*}
&0 
\to \phi^*\big (\mathcal L^{-1} \otimes
 \w \pi_1^*\mathcal  O_{\ms(n-1)}(\mathfrak D_{n-1})
 \otimes \w \pi_2^*\mathcal  O_X(2\ell) \big )
 \to \mathcal V&   \nonumber  \\
&\to R^1 (\rf)_*\rev^*\big (\pi^*  (\mathcal E_{n-1}^0)^*  
 \otimes \mathcal O_{\mathfrak B_*}(-1)\big ) \otimes 
 \phi^* \big (\w \pi_1^*\mathcal  O_{\ms(n-1)}
 (\mathfrak D_{n-1}) \otimes \w \pi_2^*\mathcal  O_X(2\ell)
 \big ) \to 0.& 
\end{eqnarray*}
\end{proposition}
\begin{proof}
Recall that $\mathfrak B_* = \Pee(\mathcal E_{n-1}^0)$. 
The kernel of the tautological surjection $\pi^*\mathcal E_{n-1}^0
\to \mathcal O_{\mathfrak B_*}(1) \to 0$ is a line bundle. 
By comparing the first Chern classes, we get 
\begin{eqnarray*}
0 \to \pi^*\mathcal L \otimes \mathcal O_{\mathfrak B_*}(-1) \to
\pi^*\mathcal E_{n-1}^0 \to \mathcal O_{\mathfrak B_*}(1) \to 0.
\end{eqnarray*}
Tensoring with $\pi^*\mathcal L^{-1} \otimes
\mathcal O_{\mathfrak B_*}(-1)$, we obtain the exact sequence
\begin{eqnarray*}
0 \to \mathcal O_{\mathfrak B_*}(-2) \to \pi^* \big ( 
\mathcal E_{n-1}^0  \otimes \mathcal L^{-1} \big ) \otimes
\mathcal O_{\mathfrak B_*}(-1) \to \pi^*\mathcal L^{-1} \to 0. 
\end{eqnarray*} 
Note that $\mathcal E_{n-1}^0  \otimes \mathcal L^{-1} \cong 
(\mathcal E_{n-1}^0)^*$. 
Applying $\rev^*$ to the above exact sequence yields
\begin{eqnarray*}
0 \to \rev^*\mathcal O_{\mathfrak B_*}(-2) \to \rev^*
\big (\pi^*  (\mathcal E_{n-1}^0)^*  \otimes
\mathcal O_{\mathfrak B_*}(-1)\big ) 
\to (\pi \circ \rev)^*\mathcal L^{-1} \to 0.
\end{eqnarray*}
By (\ref{com-diagram2}), $\pi \circ \rev = \phi \circ \rf$. 
Rewriting the 3rd term, we have
\begin{eqnarray}\label{ex-seq4}
0 \to \rev^*\mathcal O_{\mathfrak B_*}(-2) \to \rev^*
\big (\pi^*  (\mathcal E_{n-1}^0)^*  \otimes
\mathcal O_{\mathfrak B_*}(-1)\big )   
\to \rf^* \phi^*\mathcal L^{-1} \to 0.
\end{eqnarray}

Applying the functor $(\rf)_*$ to (\ref{ex-seq4}), we get 
the exact sequence 
\begin{eqnarray*}
0 
&\to& \phi^*\mathcal L^{-1}
 \to R^1(\rf)_*(\rev)^*\mathcal O_{\mathfrak B_*}(-2) \\
&\to& R^1 (\rf)_*\rev^*\big (\pi^*  (\mathcal E_{n-1}^0)^*  \otimes
 \mathcal O_{\mathfrak B_*}(-1)\big ) \to 0
\end{eqnarray*}
where we have used the projection formula, 
$(\rf)_*\mathcal O_{\mathfrak O_1} \cong 
\mathcal O_{\mathfrak O_0}$, and 
\begin{eqnarray*}
R^1(\rf)_*\mathcal O_{\mathfrak O_1} = 0, \qquad
(\rf)_*\rev^*\big (
\pi^*(\mathcal E_{n-1}^0)^* \otimes 
\mathcal O_{\mathfrak B_*}(-1)\big ) = 0.
\end{eqnarray*}
By Lemma~\ref{obstruction-bundle1} and Lemma~\ref{B*B*},
we obtain the desired exact sequence for $\mathcal V$.
\end{proof}

\begin{remark}\label{remark1}
Fix a point $(V_1, x) \in \ms(n-1) \times X$. Via 
$
\phi^{-1}(V_1, x) \cong \pms_{0, 0}(\mathbb P^1, d[\mathbb P^1]),
$
the restriction of 
$R^1 (\rf)_*\rev^*\big (\pi^* (\mathcal E_1^0)^* \otimes
\mathcal O_{\mathfrak B_*}(-1)\big )$ to $\phi^{-1}(V_1, x)$ 
is isomorphic to
$$
R^1 (f_{1, 0})_*(ev_1)^*(\mathcal O_{\mathbb P^1}(-1) \oplus 
\mathcal O_{\mathbb P^1}(-1))
$$
where by abusing notations, we still use $f_{1, 0}$ and $ev_1$ to 
denote the forgetful map and the evaluation map from 
$\overline{\frak M}_{0, 1}(\mathbb P^1, d[\mathbb P^1])$ to
$\overline{\frak M}_{0, 0}(\mathbb P^1, d[\mathbb P^1])$ and 
$\mathbb P^1$ respectively.
\end{remark}

%%
%:
%%
%%
%%
%%
%%
%%
%%
%%
%%
%%
%%
\section{\bf The virtual fundamental class $\big [\overline 
{\frak M}_{0, 1}(\pms(n), d\mathfrak f) \big ]^{\text{vir}}$}
\label{sec:virtual}

As we saw in the construction of the classes $\Xi_i$, 
$\Xi_5$ and $\Xi_6$ don't lie in $\mathfrak B_*$. The method to 
compute the virtual cycle restricted to $ev_1^{-1}(\Xi_i)$ in 
the previous section won't work for $i = 5, 6$. In this section, 
we shall employ the localization method of Kiem-Li to find 
a sufficiently small closed subset of $\pms (n)$ containing 
the image of the virtual cycle under the evaluation map $ev_1$. 
The result will be used to show the vanishing of the Gromov-Witten 
invariants $\langle \alpha \rangle_{0, d\mathfrak f}$ when 
$\alpha$ is dual to $\Xi_5, \Xi_6$. 

Let $X = \Pee^2$, and let $C_0 \subset X$ be a smooth cubic curve.
Recall that the Zariski tangent space of $\pms(n)$ at $V \in \pms(n)$
is canonically $\Ext^1(V, V)$. Therefore, the natural map 
$K_X = \mathcal O_X(-C_0) \to \mathcal O_X$ induces a meromorphic 
$2$-form $\Theta$ on the moduli space $\pms(n)$ given point-wisely by
\begin{eqnarray}  \label{def-Theta}
\Ext^1(V, V)^* \cong \Ext^1(V, V \otimes K_X) \to \Ext^1(V, V).
\end{eqnarray}
Note that $\Theta$ is holomorphic at $V \in \pms(n)$ if 
(\ref{def-Theta}) is an isomorphism.

The constructions in \cite{K-L, L-L} show that the meromorphic $2$-form 
$\Theta$ on the moduli space $\pms(n)$ induces a meromorphic homomorphism:
\begin{eqnarray}  \label{def-eta}
\eta: \mathcal E \to \mathcal O_{}
\end{eqnarray}
over $\overline {\frak M}_{0, 1}(\pms(n), d\mathfrak f)$. Here 
$\mathcal E$ is a suitable bundle on 
$\overline {\frak M}_{0, 1}(\pms(n), d\mathfrak f)$ such that 
$$
\big [\overline {\frak M}_{0, 1}(\pms(n), d\mathfrak f) 
\big ]^{\text{vir}} \in H_{2(4n-6)}(\Lambda) 
$$
where $\Lambda \subset
\overline {\frak M}_{0, 1}(\pms(n), d\mathfrak f)$ is the degeneracy 
loci consisting of points at which either the map $\eta$ is undefined 
or not surjective. 

Next, we analyze the degeneracy loci $\Lambda$.
Let $[\mu: (D; p) \to \pms(n)] \in \Lambda$. 
Since $\mu_*[D] = d\mathfrak f$, $\mu(D)$ is 
contracted to a point by the Gieseker-Uhlenbeck morphism $\Psi$.

\begin{lemma}  \label{loci}
Let $[\mu: (D; p) \to \pms(n)] \in \Lambda$, and put
\begin{eqnarray*}  
(\Psi \circ \mu)(D) = (V_k; \zeta) \in 
\ms(n-k) \times \text{\rm Sym}^k(X) \subset \uc(n)
\end{eqnarray*}
for some $1 \le k \le (n-1)$. Then, either 
$\Supp(\zeta) \cap C_0 \ne \emptyset$ or $V_k|_{C_0}$ is not stable.
\end{lemma}
\begin{proof}
Assume that $\Supp(\zeta) \cap C_0 = \emptyset$ and $V_k|_{C_0}$ is stable.
We will draw a contradiction by showing that $\eta$ is both defined and
surjective at $[\mu: (D; p) \to \pms(n)]$.

Note that $\eta$ is defined at $[\mu: (D; p) \to \pms(n)]$ if 
$\Theta$ is holomorphic along $\mu(D)$. Let $V \in \mu(D)$. Then, 
$V^{**} = V_k$ and we have 
 $0 \to V \to V_k \to Q \to 0
$
where $Q$ is a torsion sheaf with $\sum_{x \in X} h^0(X, Q_x) \, x = \zeta$.
So $V|_{C_0} \cong V_k|_{C_0}$ is locally free and stable.
From $0 \to K_X \to \mathcal O_X \to \mathcal O_{C_0} \to 0$, we get
$
0 \to V \otimes K_X \to V \to V|_{C_0} \to 0$.
Applying the functor $\text{Hom}(V, \cdot)$, we obtain a long exact sequence:
\begin{eqnarray*}  
&0 \to \text{Hom}(V, V \otimes K_X) \to \text{Hom}(V, V) \to 
   \text{Hom}(V, V|_{C_0})& \\
&\to \text{Ext}^1(V, V \otimes K_X) 
   \to \text{Ext}^1(V, V).&
\end{eqnarray*}
Since $V$ and $V_k|_{C_0}$ are stable, we have $\text{Hom}(V, V)
\cong \C$, $\text{Hom}(V, V \otimes K_X) = 0$ and $\text{Hom}(V, 
V|_{C_0}) \cong \text{Hom}(V_k|_{C_0}, V_k|_{C_0}) \cong \C$.
The above exact sequence is simplified to
$$
0 \to \text{Ext}^1(V, V \otimes K_X) \to \text{Ext}^1(V, V).
$$
Since $\dim \text{Ext}^1(V, V \otimes K_X) = \dim \text{Ext}^1(V, V)$,
we obtain an isomorphism 
$$
\text{Ext}^1(V, V \otimes K_X) \cong \text{Ext}^1(V, V).
$$
Hence the meromorphic $2$-form $\Theta$ is defined at $V \in \mu(D)$.
This proves that $\Theta$ is holomorphic along $\mu(D)$.
So $\eta$ is defined at $[\mu: (D; p) \to \pms(n)]$.

The above argument also shows that $\Theta|_{\mu(D)}$ is 
an isomorphism. Since $\mu$ is not a constant map, the image of
$\mu_*: T_{D_{\text{reg}}} \to T_{\pms(n)}$ does not lie in the
null space of $\Theta: T_{\pms(n)} \dasharrow \big ( T_{\pms(n)} \big )^*$,
where $D_{\text{reg}}$ denotes the smooth part of $D$.  
By the vanishing criterion in \cite{K-L}, $\eta$ is surjective at
$[\mu: (D; p) \to \pms(n)]$.
\end{proof}

\begin{lemma}  \label{stability-res}
Let $C_0 \subset X = \Pee^2$ be a smooth cubic curve. Let $n \ge 1$,
and let $V \in \pms(n)$ be generic. Then, the restriction 
$V|_{C_0}$ is stable.
\end{lemma}
\begin{proof}
It is well-known that the cotangent bundle $\Omega_X$ is stable. 
So $\Omega_X \otimes \mathcal O_X(1) \in \pms(1)$. Let $\xi$ consist
of $(n-1)$ distinct points away from $C_0$. Choose a surjection
$\Omega_X \otimes \mathcal O_X(1) \to \mathcal O_\xi$, 
and let $V_0$ be the kernel. Then $V_0 \in \pms(n)$ and $V_0|_{C_0} 
\cong \big ( \Omega_X \otimes \mathcal O_X(1) \big )|_{C_0}$. 
Since $\pms(n)$ is irreducible and the open subset $\ms(n)$
is nonempty, our lemma will follow if we can prove
that $\big ( \Omega_X \otimes \mathcal O_X(1) \big )|_{C_0}$ is stable.

Let $\mathcal O_{C_0}(D)$ be any sub-line-bundle of
$\big ( \Omega_X \otimes \mathcal O_X(1) \big )|_{C_0}$. Note that 
the degree of $\big ( \Omega_X \otimes \mathcal O_X(1) \big )|_{C_0}$
is $-3$, and there is an exact sequence
\begin{eqnarray}  \label{stability-res.1}
0 \to \big ( \Omega_X \otimes \mathcal O_X(1) \big )|_{C_0} \to
\mathcal O_{C_0}^{\oplus 3} \to \mathcal O_X(1) |_{C_0} \to 0
\end{eqnarray}
induced from the exact sequence $0 \to \Omega_X \to 
\mathcal O_X(-1)^{\oplus 3} \to \mathcal O_X \to 0$. Thus
\begin{eqnarray}  \label{stability-res.2}
\deg(D) \le 0.
\end{eqnarray}

If $\deg(D) = 0$, then $D$ must be the trivial divisor and 
$\mathcal O_{C_0}(D) = \mathcal O_{C_0}$. So $H^0 \big (C_0, 
\big ( \Omega_X \otimes \mathcal O_X(1) \big )|_{C_0} \big ) \ne 0$.
On the other hand, (\ref{stability-res.1}) induces an exact sequence
\begin{eqnarray*}  
0 \to H^0 \big (C_0, \big ( \Omega_X \otimes \mathcal O_X(1) 
\big )|_{C_0} \big ) \to
H^0 \big (C_0, \mathcal O_{C_0}^{\oplus 3} \big ) \overset{f_0}{\to}
H^0 \big (C_0, \mathcal O_X(1) |_{C_0} \big ).
\end{eqnarray*}
The image of $f_0$ is $H^0(X, \mathcal O_X(1))|_{C_0} =
H^0 \big (C_0, \mathcal O_X(1) |_{C_0} \big )$. So $f_0$ is surjective.
Since $H^0 \big (C_0, \mathcal O_{C_0}^{\oplus 3} \big )$ and
$H^0 \big (C_0, \mathcal O_X(1) |_{C_0} \big )$ have the same dimension,
$f_0$ is an isomorphism and $H^0 \big (C_0, 
\big ( \Omega_X \otimes \mathcal O_X(1) \big )|_{C_0} \big ) = 0$.
Hence we obtain a contradiction.

If $\deg(D) = -1$, then $\mathcal O_{C_0}(D) = \mathcal O_{C_0}(-x)$
for a unique $x \in C_0$. So the map
\begin{eqnarray*}  
H^0 \big (C_0, \mathcal O_{C_0}(x)^{\oplus 3} \big ) \overset{f_1}{\to}
H^0 \big (C_0, \mathcal O_X(1) |_{C_0} \otimes \mathcal O_{C_0}(x) \big ).
\end{eqnarray*}
induced from (\ref{stability-res.1}) is not injective. On the other hand,
since $H^0(C_0, \mathcal O_{C_0}(x)) \cong \C$, 
\begin{eqnarray*}  
   \text{\rm im}(f_1)
&=&H^0(X, \mathcal O_X(1))|_{C_0} \otimes 
       H^0 \big (C_0, \mathcal O_{C_0}(x) \big )  \\
&=&H^0 \big (C_0, \mathcal O_X(1) |_{C_0} \big ) \otimes 
       H^0 \big (C_0, \mathcal O_{C_0}(x) \big )  \\
&\subset&H^0 \big (C_0, \mathcal O_X(1) |_{C_0} \otimes 
             \mathcal O_{C_0}(x) \big ).
\end{eqnarray*}
So the linear system corresponding to $\text{\rm im}(f_1)$ consists of
all the elements $\ell|_{C_0}+x$ where $\ell$ denotes lines in $X$.
In particular, the dimension of $\text{\rm im}(f_1)$ is $3$.
Thus $f_1$ must be injective. Again, we obtain a contradiction.

By (\ref{stability-res.2}), $\deg(D) \le -2$. Therefore, 
$\big ( \Omega_X \otimes \mathcal O_X(1) \big )|_{C_0}$ is stable.
\end{proof}

\begin{definition}  \label{def:UandT}
For $n \ge 2$, we define $\mathfrak T_{C_0}(n)$ (respectively, 
$\mathfrak U_{C_0}(n)$) to be 
the subset of $\pms(n)$ consisting of all the non-locally free 
sheaves $V$ such that $V|_{C_0}$ contains torsion 
(respectively, $V|_{C_0}$ is torsion-free and unstable). 
\end{definition}

%Here is a lemma whose proof is almost obvious. 
\begin{lemma}  \label{TandU}
Let $V \in \pms(n) - \ms(n)$ and $\Psi(V) = (V_k; \zeta)$. Then,

{\rm (i)} $V \in \mathfrak T_{C_0}(n)$ if and only if 
$\Supp(\zeta) \cap C_0 \ne \emptyset$;

{\rm (ii)} $V \in \mathfrak U_{C_0}(n)$ if and only if 
$\Supp(\zeta) \cap C_0 = \emptyset$ and $V_k|_{C_0}$ is unstable. \qed
\end{lemma}
%\begin{proof}
%(i) We have the canonical double-dual exact sequence:
%\begin{eqnarray}   \label{TandU.1} 
%0 \to V \to V_k \to Q \to 0
%\end{eqnarray}
%where $Q$ is a torsion sheaf supported on $\zeta$. Note that
%${\rm Tor}^1(Q, \mathcal O_{C_0}) \ne 0$ if and only if 
%$\Supp(Q) \cap C_0 \ne \emptyset$.
%Thus $V \in \mathfrak T_{C_0}(n)$ if and only if 
%$\Supp(\zeta) \cap C_0 \ne \emptyset$.

%{\it Proof.} Let $x \in \Supp(Q) \cap C_0$. Tensoring the exact sequence
%$$0 \to R \to R \oplus R \to I_x \to 0$$
%by $\mathcal O_{C_0}$, we see that 
%${\rm Tor}^1(I_x, \mathcal O_{C_0}) = 0$. Tensoring the exact sequence
%$$0 \to I_x \to \mathcal O_X \to \mathcal O_x \to 0$$
%by $\mathcal O_{C_0}$ and noting $I_x|_{C_0} \cong \C_x \oplus 
%\mathcal O_{C_0}(-x)$, we obtain
%${\rm Tor}^1(\mathcal O_x, \mathcal O_{C_0}) = \C_x \ne 0$ and
%${\rm Tor}^2(\mathcal O_x, \mathcal O_{C_0}) \cong 
%{\rm Tor}^1(I_x, \mathcal O_{C_0}) = 0$. Using induction on the
%length of $Q$ at $x$ and the exact sequence
%$$0 \to \mathcal O_x \to Q \to Q' \to 0,$$
%we conclude that ${\rm Tor}^1(Q, \mathcal O_{C_0}) \ne 0$ and
%${\rm Tor}^2(Q, \mathcal O_{C_0}) = 0$.

%(ii) By (i), $V|_{C_0}$ is torsion-free if and only if 
%$\Supp(\zeta) \cap C_0 = \emptyset$. In this case, 
%\begin{eqnarray*}
%V|_{C_0} \cong V_k|_{C_0}.
%\end{eqnarray*}
%So $V \in \mathfrak U_{C_0}(n)$ if and only 
%if $\Supp(\zeta) \cap C_0 = \emptyset$ and $V_k|_{C_0}$ is unstable.
%\end{proof}

\begin{lemma}  \label{image_loci}
{\rm (i)} $(ev_1)(\Lambda) \subset \mathfrak T_{C_0}(n) \coprod 
\mathfrak U_{C_0}(n)$;

\smallskip
{\rm (ii)} Both $\mathfrak T_{C_0}(n)$ and $\mathfrak T_{C_0}(n) 
\coprod \mathfrak U_{C_0}(n)$ are closed subsets of $\pms(n)$. 

%\smallskip
%{\rm (iii)} $\dim \mathfrak T_{C_0}(n)
%= \dim (\mathfrak T_{C_0}(n) \coprod \mathfrak U_{C_0}(n)) = 4n-6$;

%\medskip
%{\rm (iv)} $\mathfrak T_{C_0}(n) = \overline{\Psi^{-1}(\ms(n-1) 
%\times C_0)}$ is irreducible.
\end{lemma}
\begin{proof}
By Lemma~\ref{loci}, the image of a point in $(ev_1)(\Lambda)$ 
under $\Psi$ is of the form
\begin{eqnarray*}  
(V_k; \zeta) \in \ms(n-k) \times \text{\rm Sym}^k(X) \subset \uc(n)
\end{eqnarray*}   
where $1 \le k \le (n-1)$, and either $\Supp(\zeta) \cap C_0 \ne 
\emptyset$ or $V_k|_{C_0}$ is not stable. So we see from 
Lemma~\ref{TandU} that $(ev_1)(\Lambda) \,\, \subset \,\, 
\mathfrak T_{C_0}(n) \coprod \mathfrak U_{C_0}(n)$. This proves (i).

Since being torsion-free and being stable are open conditions, 
both $\mathfrak T_{C_0}(n)$ and $\mathfrak T_{C_0}(n) \coprod 
\mathfrak U_{C_0}(n)$ are closed subsets of $\pms(n)$. This proves (ii).
\end{proof}

\section{\bf The $1$-point Gromov-Witten invariants}
\label{sec:1point}

Now we are ready to compute the $1$-point Gromov-Witten invariants
%By (\ref{expected-dim}), the expected complex dimension of 
%$\pms_{0, 1}(\pms(n), d\mathfrak f)$ is $\frak d = (4n-6)$. 
%In this section, we will compute the $1$-point Gromov-Witten invariants
\begin{eqnarray}   \label{def-1pt}
\langle \alpha \rangle_{0, d\mathfrak f} \,\,
= \int_{\big [\overline {\frak M}_{0, 1}(\pms(n), 
d\mathfrak f) \big ]^{\text{vir}}} \, ev_1^*\alpha  
\end{eqnarray}
where $\alpha \in H^{8n-12}(\pms(n); \C)$ denotes the Poincar\' e 
duals of the classes $\Xi_1, \ldots, \Xi_6 \in H_4(\pms(n); \C)$.
By abusing notations, we use $\Xi_i$ to stand for both 
the class in $H_4(\pms(n); \C)$ and its Poincar\' e dual
in $H^{8n-12}(\pms(n); \C)$.

\begin{lemma}  \label{inv-Xi_1}
$\langle \Xi_1 \rangle_{0, d\mathfrak f} = - {6/d^2}$.
\end{lemma}
\begin{proof}
Let $\alpha = \Xi_1$. By Subsect.~\ref{subsect_Xi12},
$\Xi_1 = \pi^{-1}(\{V_1\} \times \ell) = \Pee(V_1|_\ell)$ 
where the stable vector bundle $V_1 \in \ms(n-1)$ and the line
$\ell \subset X$ are fixed. Let $\ms = (ev_1)^{-1}(\Xi_1)$. 
Then $\ms \subset \mathfrak O_1$. By  
Prop.~\ref{obstruction-bundle2} and Prop.~\ref{virtual-prop} 
for $k=1$ and $\mathfrak O= \mathfrak O_1$, we obtain
\begin{eqnarray}   \label{inv-Xi_1.111}
   \langle\alpha\rangle_{0, d\mathfrak f}
&=&\int_{\big [\pms_{0, 1}(\pms(n), d\mathfrak f) \big ]^{\text{vir}}}
   \,\, (ev_1)^*\alpha    
=\int_{(\rf)^*c_{2d-1}(\mathcal V)} \,\, (ev_1)^*\alpha \nonumber \\
&=&\int_{(\rf)^*\phi^*(-\mathcal L + \w \pi_1^* \mathfrak D_{n-1}  
   + 2 \w \pi_2^*\ell) \cdot 
   (\rf)^*c_{2d-2}({\mathbb E})} \,\, (ev_1)^*\alpha  
\end{eqnarray}
where ${\mathbb E} = R^1 (\rf)_*\rev^*\big (\pi^*  (\mathcal E_1^0)^*  
\otimes \mathcal O_{\mathfrak B_*}(-1)\big ) \otimes 
\phi^* \big (\w \pi_1^*\mathcal  O_{\ms(n-1)}
(\mathfrak D_{n-1}) \otimes \w \pi_2^*\mathcal  O_X(2\ell) \big )$. 
Note that $\Xi_1 \subset \mathfrak B_*$ and $\phi \circ \rf 
= \pi \circ \rev$. So we obtain
\begin{eqnarray*}   
&  &\langle\alpha\rangle_{0, d\mathfrak f}     \\
&=&(\rev)^*\pi^*(-\mathcal L + \w \pi_1^* \mathfrak D_{n-1} 
   + 2 \w \pi_2^*\ell) \cdot 
   (\rf)^*c_{2d-2}({\mathbb E}) \cdot (\rev)^*\big ( [\Xi_1] 
   \cdot c_1(\mathcal O_{\mathfrak B_*}(\mathfrak B_*)) \big )   \\
&=&(\rf)^*c_{2d-2}({\mathbb E}) \cdot (\rev)^*\Big (
   \pi^*(-\mathcal L + \w \pi_1^* \mathfrak D_{n-1} 
   + 2 \w \pi_2^*\ell) \cdot [\Xi_1] 
   \cdot c_1(\mathcal O_{\mathfrak B_*}(\mathfrak B_*)) \Big ).
\end{eqnarray*}
Recall the definitions of $\mathcal L$ and $\mathfrak D_{n-1}$
in Proposition~\ref{obstruction-bundle2} and Lemma \ref{B*B*}. We have
\begin{eqnarray}
&\mathcal L = \det \big (\W{\mathcal E}_{n-1} \big ) = (2, -1) 
  \,\, \in \,\, \text{Pic}(\Gamma \times X)&   \label{comp_L}  \\
&\mathfrak D_{n-1} = 2a_{n-1, 1} - a_{n-1, 2} - a_{n-1, 0} = 2 
  \,\, \in \,\, \text{Pic}(\Gamma)&       \label{comp_D}  \\
&c_1(\mathcal O_{\mathfrak B_*}(\mathfrak B_*))|_{\Pee} = 
  -2 c_1 \big (\mathcal O_{\Pee}(1) \big ) + \pi^*(2, 2)
  \,\, \in \,\, \text{Pic}(\Pee)&       \label{comp_B*}
\end{eqnarray}
in view of the exact sequence (\ref{exact-P1}),
the degrees in (\ref{ank}) and Lemma~\ref{B*B*}. Thus,
\begin{eqnarray}   \label{inv-Xi_1.112}
\langle\alpha\rangle_{0, d\mathfrak f} =
(\rf)^*c_{2d-2}({\mathbb E}) \cdot (\rev)^*\Big (
\pi^*(0, 3) \cdot [\Xi_1] \cdot \big (-2 c_1(
\mathcal O_{\Pee}(1)) + \pi^*(2, 2) \big ) \Big ).
\end{eqnarray}

Since $[\Xi_1] = \pi^*(\{V_1\} \times \ell)$, 
it follows immediately that 
\begin{eqnarray}  \label{inv-Xi_1.1} 
\langle\alpha\rangle_{0, d\mathfrak f}
= -6 \cdot (\rf)^*c_{2d-2}({\mathbb E}) \cdot (\rev)^*[\xi]
\end{eqnarray}
where $[\xi]$ denotes the cycle of a fixed point $\xi$ in 
$\Pee \subset \mathfrak B_*$.

Let $\frak M_1' = (\rev)^{-1}(\xi) = (ev_1)^{-1}(\xi)$.
If $[\mu\colon (D; p) \to \pms(n)] \in \frak M_1'$, 
then $\mu(p) = \xi$ and $\pi(\mu(D)) = \pi(\mu(p)) = \pi(\xi)
= (V_1, x) \in \ms(n-1) \times X$. So $\mu(D) = \mathfrak f$
which denotes the unique fiber of $\pi$ from (\ref{def-pi})
containing the point $\xi \in \mathfrak B_*$.
Thus the restriction of the forgetful map $\rf$ to $\frak M_1'$ 
gives a degree-$d$ morphism from $\frak M_1'$ to 
$$
\frak M_0' \buildrel{\rm def}\over = \rf(\frak M_1') 
= \phi^{-1}(V_1, x).
$$ 
Hence, as algebraic cycles, we have $(\rf)_*[\frak M_1'] = 
d[\frak M_0'] = d \cdot \phi^*[(V_1, x)]$. By (\ref{inv-Xi_1.1}), 
\begin{eqnarray}  \label{inv-Xi_1.2} 
   \langle\alpha\rangle_{0, d\mathfrak f}
&=&-6 \cdot c_{2d-2}({\mathbb E}) \cdot (\rf)_*[\frak M_1']
=-6d \cdot c_{2d-2}({\mathbb E}) \cdot \phi^*[(V_1, x)]
   \nonumber   \\
&=&-6d \cdot c_{2d-2} \left ({\mathbb E}|_{\phi^{-1}(V_1, x)} 
   \right ). 
\end{eqnarray}

By Remark~\ref{remark1}, ${\mathbb E}|_{\phi^{-1}(V_1, x)} 
\cong R^1 (f_{1, 0})_*(ev_1)^*(\mathcal O_{\mathbb P^1}(-1) \oplus 
\mathcal O_{\mathbb P^1}(-1))$ where $f_{1, 0}$ and $ev_1$ denote 
the forgetful map and the evaluation map from the moduli space 
$\overline{\frak M}_{0, 1}(\mathbb P^1, d[\mathbb P^1])$ to
$\overline{\frak M}_{0, 0}(\mathbb P^1, d[\mathbb P^1])$ 
and $\mathbb P^1$ respectively. We have
$$
c_{2d-2} \left ( R^1 (f_{1, 0})_*
(ev_1)^*(\mathcal O_{\mathbb P^1}(-1)
\oplus \mathcal  O_{\mathbb P^1}(-1)) \right ) = {1 \over d^3}
$$
by the Theorem 9.2.3 in \cite{C-K}. Therefore, 
$\langle\alpha\rangle_{0, d\mathfrak f} = -6/d^2$
by (\ref{inv-Xi_1.2}).
\end{proof}

For $i = 2, 3$ or $4$,  we may assume that the classes $\Xi_2$, $\Xi_3$ and $\Xi_4$ 
are represented by complex surfaces in $\Pee \subset \mathfrak B_*$.
The same proofs of (\ref{inv-Xi_1.112}) and (\ref{inv-Xi_1.1})
show that 
\begin{eqnarray*}
\langle \Xi_i \rangle_{0, d\mathfrak f}
= a_i \cdot (\rf)^*c_{2d-2}({\mathbb E}) \cdot (\rev)^*[\xi]
\end{eqnarray*}
where $[\xi]$ denotes the cycle of a fixed point $\xi$ in 
$\Pee \subset \mathfrak B_*$, and 
\begin{eqnarray}    \label{a234}
a_i = \pi^*(0, 3) \cdot [\Xi_i] \cdot \big (-2 c_1(
    \mathcal O_{\Pee}(1)) + \pi^*(2, 2) \big )
\end{eqnarray}
is the intersection number in $\Pee$. Note from the last two 
paragraphs in the proof of Lemma~\ref{inv-Xi_1} that
$(\rf)^*c_{2d-2}({\mathbb E}) \cdot (\rev)^*[\xi] = 1/d^2$.
Therefore, 
\begin{eqnarray}  \label{setup}
\langle \Xi_i \rangle_{0, d\mathfrak f}
= {a_i \over d^2}.
\end{eqnarray}

\begin{theorem}  \label{thm:1pt-inv}
Let $d \ge 1$ and $n \ge 3$. The Gromov-Witten invariants 
$\langle \alpha \rangle_{0, d\mathfrak f}$ for the classes
$\alpha = \text{\rm PD}(\Xi_1), \ldots, \text{\rm PD}(\Xi_6) 
\in H^{8n-12}(\pms(n); \C)$ are respectively equal to
\begin{eqnarray*}   
- {6/d^2}, \quad {12/d^2}, \quad 0, \quad 
- {6/d^2}, \quad 0, \quad 0.
\end{eqnarray*}
\end{theorem}
\begin{proof} 
First of all, $\langle \Xi_1 \rangle_{0, d\mathfrak f} 
= {-6/d^2}$ is Lemma \ref{inv-Xi_1}.

Next, $\langle \Xi_2 \rangle_{0, d\mathfrak f} = {12/d^2}$  
follows from the computation of the number in (\ref{a234}):
\begin{eqnarray*}
a_2 = 3\pi^*\ell \cdot c_1(\mathcal O_{\Pee(V_1)}(1)) \cdot 
\big (-2 c_1(\mathcal O_{\Pee(V_1)}(1)) + 2\pi^*\ell \big )
= 12
\end{eqnarray*}
by (\ref{Bsquare}),
where $\pi: \Pee(V_1) \to X$ denotes the tautological projection. 

Since $\Xi_3 = \pi^{-1}(\Gamma \times \{y\})$ with $y \in X$,  
$\langle \Xi_3 \rangle_{0, d\mathfrak f} = 0$ comes from the computation
\begin{equation}
a_3 = \pi^*(0, 3) \cdot [\Xi_3] \cdot \big (-2 c_1(
\mathcal O_{\Pee}(1)) + \pi^*(2, 2) \big ) = 0.  
\end{equation}

Similary, $\langle \Xi_4 \rangle_{0, d\mathfrak f} = - {6/d^2}$ 
follows from the computation 
\begin{eqnarray*} 
a_4 = \pi^*(0, 3) \cdot c_1(\mathcal O_W(1)) \cdot 
\big (-2 c_1(\mathcal O_W(1)) + \pi^*(2, 2) \big ) = -6
\end{eqnarray*}
by (\ref{Bsquare2}), where $W = \pi^{-1}(\Gamma \times \ell) 
\subset \Pee \subset \pms(n)$ for a fixed line $\ell \subset X$. 

To prove $\langle \Xi_5 \rangle_{0, d\mathfrak f} = 0$, 
choose the vector bundle $V_2 \in \ms(n-2)$ and the distinct points 
$x_1, x_2 \in X$ in Subsect.~\ref{subsect_Xi_5} such that 
$V_2|_{C_0}$ is stable and $x_1, x_2 \not \in C_0$. 
By Lemma~\ref{TandU}, $\Xi_5 \cap \big (\mathfrak T_{C_0}(n) \cup 
\mathfrak U_{C_0}(n) \big ) = \emptyset$. Hence
$\Xi_5 \cap (ev_1)(\Lambda) = \emptyset$
by Lemma~\ref{image_loci}~(i).
Since $\big [\overline {\frak M}_{0, 1}(\pms(n), d\mathfrak f) 
\big ]^{\text{vir}} \in H_{2(4n-6)}(\Lambda)$,  we get
\begin{equation}
  \langle \Xi_5 \rangle_{0, d\mathfrak f} \,\,
= \int_{\big [\overline {\frak M}_{0, 1}(\pms(n), 
    d\mathfrak f) \big ]^{\text{vir}}} \, ev_1^*\Xi_5  
= ev_{1*}\big [\overline {\frak M}_{0, 1}(\pms(n), d\mathfrak f) 
    \big ]^{\text{vir}}  \cdot \Xi_5  
= 0.
\end{equation}

Finally, to prove $\langle \Xi_6 \rangle_{0, d\mathfrak f} = 0$, 
choose the vector bundle $V_2 \in \ms(n-2)$ and the point $x \in X$ 
in Subsect.~\ref{subsect_Xi_6} such that $V_2|_{C_0}$ is stable and 
$x \not \in C_0$. Now our result follows from the same proof  
in the previous paragraph.
\end{proof}

\begin{remark}\label{rmk:push-vfc} 
Let $d \ge 1$ and $n \ge 3$. Using the Theorem above, one can show that
\begin{eqnarray*} 
ev_{1*}\big [\overline {\frak M}_{0, 1}(\pms(n), d\mathfrak f) 
\big ]^{\text{vir}}  \,\, = \,\, 1/d^2 \,\, \mathfrak T_{C_0}(n).
\end{eqnarray*}
\end{remark}


\begin{thebibliography}{ABCD}

\bibitem[Bea]{Bea} A. Beauville,
{\em Sur la cohomologie de certains espaces de modules de fibr\'es
vectoriels}, Geometry and Analysis (Bombay, 1992), 37-40,
Tata Inst. Fund. Res., Bombay, 1995.

\bibitem[Beh]{Beh} K. Behrend,
{\em Gromov-Witten invariants in algebraic geometry},
 Invent. Math. {\bf 127} (1997) 601-617.    

\bibitem[B-F]{B-F} K. Behrend, B. Fantechi,
{\em The intrinsic normal cone},
 Invent. Math. {\bf 128} (1997) 45-88. 

%\bibitem[B-G]{B-G} J. Bryan, T. Graber,
%{\em The crepant resolution conjecture}. Preprint.

%\bibitem[C-Y]{C-Y} T. Coates, Y. Ruan,
%{\em Quantum cohomology and crepant resolutions: a conjecture}. 
%Preprint.

\bibitem[C-K]{C-K} D. Cox, S. Katz, 
{\em Mirror symmetry and algebraic geometry},
{\rm Mathematical Surveys and Monographs} {\bf 68}, 
Amer. Math. Soc., Providence, RI (1999). 

%\bibitem[ELQ]{ELQ} D. Edidin, W.-P. Li, Z. Qin,
%{\em Gromov-Witten invariants of the Hilbert scheme of $3$-points 
%on ${\mathbb P}^2$}. Asian J. Math. {\bf 7} (2003), 551-574. 

\bibitem[E-S]{E-S} G. Ellingsrud, S. Str\o mme,
{\em Towards the Chow ring of the Hilbert scheme of $\mathbb
P^2$}, J. reine angew. Math. {\bf 441} (1993), 33-44.

%\bibitem[E-L]{E-L} G. Ellingsrud, M. Lehn,
%{\em Irreducibility of the punctual quotient scheme of a surface},
%Ark. Mat. {\bf 37} (1999), 245-254.

\bibitem[Ful]{Ful} W. Fulton,
{\em Intersection Theory}, Ergebnisse der Mathematik und 
ihrer Grenz\-gebiete 3. Folge {\bf 2}. Springer,  
Berlin Heidelberg  New York Tokyo, 1994.

\bibitem[F-P]{F-P} W. Fulton, R. Pandharipande, 
{\em Notes on stable maps and quantum cohomology}. 
Algebraic Geometry---Santa Cruz 1995, 45-96, 
Proc. Sympos. Pure Math. {\bf 62}, 
Amer. Math. Soc., Providence, RI (1997).

\bibitem[Get]{Get} E. Getzler,
{\em Intersection theory on $\overline M_{1,4}$ and 
elliptic Gromov-Witten invariants},
J. AMS {\bf 10} (1997) 973-998.


\bibitem[Gro]{Gro} A. Grothendieck,
{\em \'Etude Cohomologique des faisceaux coh\'erents}, EGA III, Publ. IHES, No. {\bf 17} (1963).

%\bibitem[Har]{Har} R. Hartshorne,
%{\em Algebraic geometry}. Springer, Berlin-Heidelberg-New York, 1978.

\bibitem[H-H]{H-H} A. Hirschowitz, K. Hulek,
{\em Complete families of stable vector bundles over $\mathbb P_2$}.
With an appendix by K. Hulek and S. A. Str\o mme. 
Lecture Notes in Math. {\bf 1194}, Complex Analysis and 
Algebraic Geometry (Gottingen, 1985), 19-40, Springer, Berlin, 1986. 

\bibitem[K-L]{K-L} Y. Kiem, J. Li, 
{\em Gromov-Witten invariants of varieties with holomorphic $2$-forms}.
Preprint.

%\bibitem[L-S]{L-S} M. Lehn, C. Sorger,
%{\em The cup product of the Hilbert scheme for $K3$ surfaces},
%Invent. Math. {\bf 152} (2003), 305-329.

\bibitem[LJ1]{LJ1} J. Li, 
{\em Algebraic geometric interpretation of Donaldson's 
polynomial invariants}, J. Differ. Geom. {\bf 37} (1993), 
417-466. 

\bibitem[LJ2]{LJ2} J. Li,
{\em Kodaira dimension of moduli space of vector
bundles on surfaces}, Invent. Math. {\bf 115} (1994), 1-40.

\bibitem[L-L]{L-L} J. Li, W.-P. Li,
{\em Two point extremal Gromov-Witten invariants of Hilbert
schemes of points on surfaces}. Preprint.

\bibitem[LT1]{LT1} J. Li, G. Tian, 
{\em Virtual moduli cycles and Gromov-Witten invariants of 
algebraic varieties}, J. A.M.S. {\bf 11} (1998) 19-174.

\bibitem[LT2]{LT2} J. Li, G. Tian, 
{\em Virtual moduli cycles and Gromov-Witten invariants of 
general symplectic manifolds}, Topics in symplectic 
$4$-manifolds (Irvine, CA, 1996), 
First Int. Press Lect. Ser., I, Internat.
Press, Cambridge, MA, (1998) 47--83.

\bibitem[L-Q]{L-Q} W.-P. Li, Z. Qin,
{\em On $1$-point Gromov-Witten invariants of the Hilbert 
schemes of points on surfaces}. Proceedings of 
8th G\" okova Geometry-Topology Conference (2001).
Turkish J. Math. {\bf 26} (2002), 53-68.

\bibitem[Mar]{Mar} E. Markman,
{\em Integral generators for the cohomology ring of moduli 
spaces of sheaves over Poisson surfaces}. 
Adv. Math. {\bf 208} (2007), 622-646.

%\bibitem[MNOP]{MNOP} D. Maulik, N. Nekrasov, A. Okounkov, R. Pandharipande,
%{\em Gromov-Witten theory and Donaldson-Thomas theory, II}. Compos. Math.
%{\bf 142} (2006), 1286-1304. 

%\bibitem[MO1]{MO1} D. Maulik, A. Oblomkov,
%{\em Quantum cohomology of the Hilbert scheme of points on $A_n$-resolutions}. 
%Preprint. 

%\bibitem[MO2]{MO2} D. Maulik, A. Oblomkov,
%{\em Donaldson-Thomas theory of $A_n \times P^1$}. Preprint. 

\bibitem[Mor]{Mor} J.W. Morgan, 
{\em Comparison of the Donaldson polynomial
invariants  with their alge\-bro-geo\-metric analogues}, 
Topology {\bf 32} (1993), 449-488.

%\bibitem[O-P]{O-P} A. Okounkov, R. Pandharipande,
%{\em Quantum cohomology of the Hilbert schemes of points in the
%plane}. Preprint.

%\bibitem[Ruan]{Ruan} Y. Ruan,
%{\em The cohomology ring of crepant resolutions of orbifolds}. 
%Contemp. Math. {\bf 403} (2006), 117–126.

%\bibitem[Q-W]{Q-W} Z. Qin, W. Wang, 
%{\em Hilbert schemes and symmetric products: a dictionary}, 
%Contemp. Math. {\bf 310} (2002), 233-257.

\bibitem[Q-Z]{Q-Z} Z. Qin, Q. Zhang, 
{\em On the crepancy of the Gieseker-Uhlenbeck morphism}. 
Asian J. Math. {\bf 12} (2008), 213-224.

\bibitem[Str]{Str} S.A. Str\o mme, 
{\em Ample divisors on fine moduli spaces on the projective plane},
Math. Z. {\bf 187} (1984), 405-423.

\bibitem[Uhl]{Uhl} K. Uhlenbeck, 
{\em Removable singularity in Yang-Mills fields}, 
Comm. Math. Phys. {\bf 83} (1982), 11-29.

\bibitem[Yos]{Yos} K. Yoshioka, 
{\em The Betti numbers of the moduli space of stable sheaves 
of rank $2$ on $\Bbb P^2$}, J. reine angew. Math. 
{\bf 453} (1994), 193-220.

\end{thebibliography}
\end{document}